\documentclass[a4paper,11pt,reqno]{amsart}

\usepackage{amssymb}
\usepackage[all]{xy}

\usepackage{hyperref}
\usepackage{xcolor}
\usepackage{enumitem}

\usepackage{verbatim}

\usepackage{tikz}
\usetikzlibrary[topaths]

\newcount\mycount

\numberwithin{equation}{section}

\numberwithin{equation}{section}
\newtheorem{theorem}{Theorem}[section]

\newtheorem{lemma}[theorem]{Lemma}
\newtheorem{lem}[theorem]{Lemma}
\newtheorem{proposition}[theorem]{Proposition}
\newtheorem{prop}[theorem]{Proposition}
\newtheorem{corollary}[theorem]{Corollary}

\theoremstyle{definition}
\newtheorem{definition}[theorem]{Definition}

\newtheorem{dfn}[theorem]{Definition}

\theoremstyle{remark}
\newtheorem{example}[theorem]{Example}

\theoremstyle{remark}
\newtheorem{rmk}[theorem]{Remark}
\newtheorem{remark}[theorem]{Remark}
\newtheorem{question}[theorem]{Question}

\newcommand\bp{\begin{proof}}
\newcommand\ep{\end{proof}}

\newcommand\Tr{\operatorname{Tr}}

\newcommand{\hotimes}{\otimes_h}
\newcommand{\cL}{\mathcal{L}}

\newcommand{\cH}{\mathcal{H}}

\newcommand{\cB}{\mathcal{B}}

\newcommand{\minotimes}{\otimes_{{\rm min}}}

\newcommand{\cN}{\mathcal{N}}

\newcommand{\cR}{\mathcal{R}}

\newcommand{\boldv}{{\bf v}}
\newcommand{\boldw}{{\bf w}}

\newcommand{\cQ}{\mathcal{Q}}

\newcommand{\Cliq}{\textsf{Cliq}}

\newcommand{\Link}{\textsf{Link}}

\newcommand{\Comm}{\textsf{Comm}}
\newcommand{\Diag}{{\rm Diag}}

\newcommand{\rmC}{\textrm{C}}
\newcommand{\rmR}{\textrm{R}}

\newcommand{\Dom}{{\rm Dom}}

\newcommand{\RWS}{\mathbb{R}_{>0}^{(W,S)}}

\newcommand{\CWS}{\mathbb{C}^{(W,S)}}

\newcommand{\OneWS}{ \{-1, 1 \}^{(W,S)}}

\begin{document}

 \title[Graph product Khintchine inequalities and Hecke C$^\ast$-algebras]{Graph product Khintchine inequalities and Hecke C$^\ast$-algebras:  Haagerup inequalities, (non)simplicity, nuclearity and exactness}

\author{Martijn Caspers, Mario Klisse}

\address{TU Delft, EWI/DIAM,
	P.O.Box 5031,
	2600 GA Delft,
	The Netherlands}

\email{m.p.t.caspers@tudelft.nl}

\email{m.klisse@tudelft.nl}

\author{Nadia S. Larsen}

\address{Department of Mathematics, University of Oslo, P.O. Box 1053, Blindern, NO-0316}

\email{nadiasl@math.uio.no}

\date{18 December 2019. {\it MSC2010}: 46L07, 46L09, 47A63.  MC is supported by the NWO Vidi grant `Non-commutative harmonic analysis and rigidity of operator algebras', VI.Vidi.192.018.  MK is supported by the NWO project `The structure of Hecke-von Neumann algebras', 613.009.125. NL acknowledges support from the Trond Mohn Foundation through the project ``Pure Mathematics in Norway''.}

\maketitle
\begin{abstract}
Graph products of groups were introduced by Green in her thesis \cite{Green}. They have an operator algebraic counterpart introduced and explored in \cite{CaspersFima}. In this paper we prove Khintchine type inequalities for general C$^{\ast}$-algebraic graph products which generalize results by Ricard and Xu \cite{RicardXu} on free products of C$^{\ast}$-algebras. We apply these inequalities in the context of (right-angled) Hecke C$^{\ast}$-algebras, which are deformations of the group algebra of Coxeter groups (see \cite{Da}). For these we deduce a Haagerup inequality which generalizes results from \cite{HaagerupExample}. We further use this to study the simplicity and trace uniqueness of (right-angled) Hecke C$^{\ast}$-algebras. Lastly we characterize exactness and nuclearity of general Hecke C$^{\ast}$-algebras.
\end{abstract}


\section*{Introduction}

A graph product of groups, first introduced in \cite{Green}, is a group theoretic construction that generalizes both free products and Cartesian products of groups. It associates to a simplicial graph with discrete groups as vertices a new group by taking the free product of the vertex groups, with added relations depending on the graph. The construction preserves many important group theoretical properties (see \cite{HaagerupProperty}, \cite{Alternatives}, \cite{CaspersIDAQP}, \cite{Ordering}, \cite{Sofic}, \cite{RapidDecay}, \cite{Green}, \cite{Linear}, \cite{Algorithm}) and covers for instance right-angled Coxeter groups and right-angled Artin groups.

In \cite{CaspersFima} Fima and the first-named author explored the operator algebraic counterpart of graph products and developed the theory of reduced and universal graph products of C$^{\ast}$-algebras as well as graph products of von Neumann algebras and quantum groups. The construction generalizes operator algebraic free products and - just as for groups - preserves properties like exactness for reduced graph products of C$^{\ast}$-algebras and the Haagerup property for von Neumann algebras. Important examples of graph products are right-angled Hecke algebras (see below), special cases of mixed $q$-Gaussian algebras  \cite{BozejkoSpeicher} as well as several group C$^\ast$-algebras that have been studied in the literature, see in particular \cite{Dykema} and Section \ref{Sect=FreeAbelian}. In a different direction, C$^{\ast}$-algebras associated to graph products of groups with distinguished positive cones were studied in \cite{CrispLaca}.

In this paper we prove a Khintchine inequality for general C$^{\ast}$-algebraic graph products. These are inequalities which estimate the operator norm of a reduced operator of a given length with the norm of certain Haagerup tensor products of column and row Hilbert spaces. In the case of free groups and words of length 1 these inequalities go back to Haagerup's fundamental paper \cite{HaagerupExample}.
 In the case of general free products and arbitrary length a Khintchine type inequality has been proven by Ricard and Xu in \cite[Section 2]{RicardXu}. In the current paper we obtain a Khintchine inequality for general graph products. We do this by introducing an intertwining technique between graph products and free products.

\begin{theorem} \label{Theorem 1}
Let $\Gamma$ be a finite simplicial graph. Consider a graph product  $\left(A,\varphi\right)=\ast_{v,\Gamma}\left(A_{v},\varphi_{v}\right)$ of unital C$^{\ast}$-algebras with \emph{GNS}-faithful states $\varphi_{v}$. Denote by $\chi_{d}$ the word length projection of length $d$. Then for every $d\in\mathbb{N}_{\geq 1}$ there exists some operator space $X_{d}$ and maps \begin{eqnarray} \nonumber j_{d}\text{: }\chi_{d}\left(A\right)\rightarrow X_{d}\text{, }\qquad\pi_{d}\text{: }\text{Dom}\left(\pi_{d}\right)\subseteq X_{d}\rightarrow\chi_{d}\left(A\right), \end{eqnarray} with $\text{Dom}\left(\pi_{d}\right)=j_{d}\left(\chi_{d}\left(A\right)\right)$ such that the following statements hold:

\begin{enumerate}[label=(\roman*)]
\item $X_d$ is a direct sum of Haagerup tensor products of column and row Hilbert spaces and is defined in \eqref{Eqn=Xd2};
\item $\pi_{d}\circ j_{d}$ is the identity on $\chi_{d}\left(A\right)$;
\item\label{Item=BoundPi} $\left\Vert \pi_{d}\text{: }\text{Dom}\left(\pi_{d}\right)\rightarrow A\right\Vert _{cb}\leq C d$ for some (explicit) constant $C$, depending only on the graph $\Gamma$.
\end{enumerate}
\end{theorem}

\hyphenation{in-equa-li-ty}

In Theorem \ref{Theorem 1} the bound on $\pi_d$ in property \ref{Item=BoundPi} should be regarded as the Khintchine inequality.
Inequalities of this kind have a wide range of applications. We refer in particular to the weak amenability results by Ricard and Xu \cite{RicardXu},  the early connections to Coxeter groups by Bo\. zejko and Speicher \cite{BozejkoSpeicher} and Nou's result on non-injectivity of $q$-Gaussians \cite{Nou}. The second half of this paper gives further applications in the case of (right-angled) Hecke C$^{\ast}$-algebras.

\vspace{0.3cm}

Hecke algebras are deformations of group algebras of a Coxeter group $W$ depending on a multi-parameter $q$ and have been studied since the 1950s. Their development played an important role in representation theory of algebraic groups (see e.g. \cite{IwahoriMatsumoto}, \cite{Bernstein}).
Hecke algebras naturally act on the Hilbert space $\ell^{2}\left(W\right)$ and thus complete to C$^{\ast}$-algebras (resp. von Neumann algebras) denoted by $C_{r,q}^{\ast}\left(W\right)$ (resp. $\mathcal{N}_{q}\left(W\right)$) which carry a canonical tracial state. In the case of spherical or affine Coxeter groups these operator algebras have been studied early in \cite{Matsumoto}, \cite{KL}, \cite{Lusztig}.   Much later, motivated by the study of weighted $L^2$-cohomology of Coxeter groups, the study of general (non-affine) Hecke  von Neumann algebras was initiated by  Dymara in \cite{Dymara1}   (see also \cite{Da}). Other relevant references are \cite{Dymara2}, \cite{Gar}, \cite{CaspersAPDE}, \cite{CSW} and \cite{RaumSkalski}.

As observed in \cite{CaspersAPDE}, Hecke C$^{\ast}$-algebras (resp. Hecke-von Neumann algebras) of right-angled Coxeter systems $\left(W,S\right)$ can be realized as graph products of finite-dimensional abelian C$^{\ast}$-algebras (resp. von Neumann algebras). Applying Theorem \ref{Theorem 1} we deduce a Haagerup inequality for right-angled Hecke algebras that generalizes Haagerup's inequality for free groups, see \cite{HaagerupExample} and \cite[Section 9.6]{Pisier}. These are inequalities that estimate the operator norm of an operator of length $d$ with the $L^{2}$-norm up to some polynomial bound.

\begin{theorem} \label{Theorem 2}
Let $\left(W,S\right)$ be a right-angled Coxeter group with finite generating set $S$. Then for every multi-parameter $q$ and $x\in\chi_{d}\left(C_{r,q}^{\ast}\left(W\right)\right), d \in \mathbb{N}_{\geq 1}$ we have \begin{eqnarray} \nonumber \left\Vert x\right\Vert \leq Cd\left\Vert x\right\Vert _{2} \end{eqnarray} for some (explicit) constant depending only on $q$ and the graph $\Gamma$.
\end{theorem}

The main motivation for the second part of this paper is the main result in \cite{Gar} where it was shown that single parameter Hecke-von Neumann algebras of irreducible right-angled Coxeter systems $\left(W,S\right)$ with $\left|S\right|\geq3$ are factors, up to a $1$-dimensional direct summand.

\begin{theorem} \cite[Garncarek]{Gar} \label{Thm=Garncarek}
Let $\left(W,S\right)$ be an irreducible right-angled Coxeter system with $\left|S\right|\geq3$. Then the single parameter Hecke-von Neumann algebra $\mathcal{N}_{q}\left(W\right)$ is a factor if and only if $q\in\left[\rho,\rho^{-1}\right]$ where $\rho$ is the radius of convergence of the spherical growth series $\sum_{w\in W}z^{\left|w\right|}$ of $W$. Moreover, for $q$ outside this interval, $\mathcal{N}_{q}\left(W\right)$ is a direct sum of a factor and $\mathbb{C}$.
\end{theorem}

We take this theorem as a motivation to study the simplicity and the uniqueness of the tracial state of (right-angled) Hecke C$^{\ast}$-algebras. By applying Theorem \ref{Theorem 2} in combination with an averaging argument inspired by \cite{Powers} we are able to give some partial answers. Note that in the single parameter case $q=1$ this has been done in \cite{DeLaHarpe} (see also \cite{Fe} and \cite{Cornulier}).

\begin{theorem} \label{Theorem 3}
Let $\left(W,S\right)$ be an irreducible Coxeter system and $q$ some multi-parameter.
\begin{enumerate}
\item If $\left(W,S\right)$ is of spherical or affine type, then $C_{r,q}^{\ast}\left(W\right)$ is not simple and does not have a unique tracial state for any choice of the parameter $q$;\\
\item If $\left(W,S\right)$ is of non-affine type and the multi-parameter $\left(q_{s}^{\epsilon_{s}}\right)_{s\in S}$ with   \begin{eqnarray}
\nonumber
\epsilon_s:= \begin{cases}
+1, & \text{if }q_s \leq 1\\
-1, & \text{if }q_s >1
\end{cases}
\end{eqnarray} lies in the closure of the region of convergence of the multivariate growth series of $W$, then $C_{r,q}^{\ast}\left(W\right)$ is not simple and does not have unique tracial state;\\
\item If $\left(W,S\right)$ is right-angled with $\left|S\right|\geq3$, then there exists an open neighbourhood $\mathcal{U}$ of the multi-parameter $1\in\mathbb{R}^{S}_{>0}$ such that $C_{r,q}^{\ast}\left(W\right)$ is simple for all $q\in\mathcal{U}$.
\end{enumerate}
\end{theorem}

In the case of free products of abelian Coxeter groups we further give a full answer to the simplicity question by making use of Dykema's results \cite{Dykema} on the simplicity of free products of finite dimensional abelian C$^{\ast}$-algebras. If the Hecke C$^\ast$-algebra has a unique trace then these simplicity results imply factoriality of the Hecke-von Neumann algebra, see the Remarks \ref{Rmk=Factor} and \ref{Rmk=Factor2}. We thus extend the range of $q$ for which factoriality of $\cN_q(W)$ holds. In particular we obtain first results in the multi-parameter case which gives a partial answer to Question 2 in \cite{Gar}.  In the case of free products of abelian Coxeter groups our answer fully settles this question.

In the final part of this paper, we prove that Hecke C$^{\ast}$-algebras are exact and characterize their nuclearity in terms of the properties of the underlying group. In particular, we show the following.

\begin{theorem} \label{Theorem 4}
Let $\left(W,S\right)$ be a Coxeter system and $q$ some multi-parameter. Then $C_{r,q}^{\ast}\left(W\right)$ is exact. Further,   $C_{r,q}^{\ast}\left(W\right)$ is nuclear if and only if $(W,S)$ is of spherical or affine type.
\end{theorem}

\subsection*{\normalfont \emph{Structure}}

In Section \ref{notation} we recall the general theory of graph products, Coxeter groups and introduce multi-parameter Hecke algebras (resp. their operator algebras). In Section \ref{A graph product Khintchine-Haagerup inequality} we obtain the Khintchine inequality from Theorem \ref{Theorem 1}. The consequences of this will be collected in Section \ref{Inequalities} where we reformulate this Khintchine inequality in the setting of right-angled Hecke C$^{\ast}$-algebras and deduce the Haagerup inequality from Theorem \ref{Theorem 2}. Section \ref{Sect=Iso} discusses isomorphism properties of Hecke algebras which are known to experts, see for instance \cite{Matsumoto}, \cite{ScottOkun}. In Section \ref{Sect=Simplicity} we will then study the simplicity and the uniqueness of the tracial state of Hecke C$^{\ast}$-algebras. In the case of spherical and affine type Coxeter systems and for free products of abelian groups we give full answers, whereas in the case of irreducible right-angled Coxeter systems we give partial answers (see Theorem \ref{Theorem 3}) by making use of the results in Section \ref{Inequalities}. Finally, Section \ref{FurtherProperties} characterizes exactness and nuclearity of Hecke C$^{\ast}$-algebras (see Theorem  \ref{Theorem 4})  and we give (counter)examples regarding the isomorphism properties from Section \ref{Sect=Iso}.


\subsection*{Acknowledgements} The authors thank Erik Opdam and Maarten Solleveld for very useful communication on isomorphisms of Hecke algebras.


\section{Preliminaries and notation} \label{notation}


\subsection{General notation}
We use the notation $\mathbb{N}:=\left\{ 1,2,\ldots\right\}$ and $\mathbb{N}_{\geq0}:=\left\{ 0,1,2,\ldots\right\}$. By $M_{k,l}(\mathbb{C})$ we denote the $k$ times $l$ matrices and we let $M_k(\mathbb{C}) := M_{k,k}(\mathbb{C})$.
We let $\mathbb{Z}_2$ be the group with two elements.
We write $\delta(\mathcal{P})$ for the function that is 1 if a statement $\mathcal{P}$ is true and  which is 0 otherwise. We write $\cB( \mathcal{H})$   for the bounded  operators on a Hilbert space $\mathcal H$.

For standard results on von Neumann algebras we refer to  \cite{Takesaki}.   For a tracial von Neumann algebra $\left(M,\tau \right)$ with faithful $\tau$ we denote the norm induced by the tracial state by $\left\Vert x \right\Vert_2 := \tau(x^\ast x)^{\frac{1}{2}}$. The following lemma is standard and shall be used several times.

\begin{lem}\label{Lem=StandardExt}
Let $A$ and $B$ be unital C$^\ast$-algebras with respective faithful traces $\tau_A$ and $\tau_B$. Let $A_0 \subseteq A$ and $B_0 \subseteq B$ be dense $\ast$-subalgebras of $A$ and $B$ and let $\pi: A_0 \rightarrow B_0$ be a $\ast$-isomorphism such that $\tau_B \circ \pi = \tau_A$. Then $\pi$ extends to a $\ast$-isomorphism $A \rightarrow B$ as well as $\pi_{\tau_A}(A)'' \rightarrow \pi_{\tau_B}(B)''$ where $\pi_{\tau_A}$ and $\pi_{\tau_B}$ are the GNS-representations.
\end{lem}
\begin{proof}
This is essentially \cite[Theorem 5.1.4]{Murphy}.
Without loss of generality we may assume that $A$ and $B$ are represented on their GNS-spaces $L^2(A, \tau_A)$ and $L^2(B, \tau_B)$ with cyclic vectors $\Omega_A$ and $\Omega_B$. Then $U:  L^2(A, \tau_A) \rightarrow L^2(B, \tau_B):  a \Omega_A \mapsto \pi(a) \Omega_B, a \in A_0$ is unitary and $\pi(a) = U a U^\ast$. Therefore $\pi$ extends to a map $A \rightarrow B$ as well as $A'' \rightarrow B''$.
\end{proof}

We use the symbol $\otimes$ for the reduced tensor product of C$^\ast$-algebras (through Takesaki's theorem \cite[Section 6.4]{Murphy} also known as the spatial or minimal tensor product) and $\overline{\otimes}$ for the tensor product of von Neumann algebras. We denote the Haagerup tensor product of operator spaces by $\otimes_h$ (see below).

\subsection{Column and row Hilbert spaces}
For the theory of operator spaces we refer to \cite{EffrosRuan} and \cite{Pisier}.  The column Hilbert space of dimension $k$ will be denoted by $\rmC_k$ and the row Hilbert space of dimension $k$ will be denoted by $\rmR_k$. $\rmC_k$ is the operator space spanned by the matrix units $e_{i,0}, i =1, \ldots, k$ in $M_k(\mathbb{C})$. Its operator space structure is the restriction of the operator space structure of $M_k(\mathbb{C})$ as a C$^\ast$-algebra. Concretely,
\[
\Vert \sum_{i=1}^k x_i \otimes e_{i, 0} \Vert_{M_l(\mathbb{C}) \otimes \rmC_k} =  \Vert \sum_{i=1}^k x_i^\ast x_i \Vert^{\frac{1}{2}}, \qquad x_i \in M_l(\mathbb{C}), l \in \mathbb{N}.
\]
Similarly, $\rmR_k$ is the operator space spanned by matrix units $e_{0,i}, i =1, \ldots, k$ in $M_k(\mathbb{C})$ and it inherits the operator space structure of  $M_k(\mathbb{C})$. We have,
\[
\Vert \sum_{i=1}^k x_i \otimes e_{0,i} \Vert_{M_l(\mathbb{C}) \otimes \rmR_k} =  \Vert \sum_{i=1}^k x_i x_i^\ast \Vert^{\frac{1}{2}}, \qquad x_i \in M_l(\mathbb{C}), l \in \mathbb{N}.
\]
For the Haagerup tensor product of operator spaces we refer to \cite[Section 9]{EffrosRuan}. We shall mainly need the completely isometric identifications (see \cite[Proposition 3.9.4 and 3.9.5]{EffrosRuan}) for $k,l \in \mathbb{N}$,
\[
\rmC_k \hotimes \rmC_l \simeq \rmC_{k+l}, \qquad \rmR_k \hotimes \rmR_l \simeq \rmR_{k+l}, \qquad \rmC_k \hotimes \rmR_l \simeq M_{k,l}(\mathbb{C}).
\]

\subsection{Graphs} \label{Graphs}

Let $\Gamma$ be a {\bf simplicial graph} with vertex set $V\Gamma$ and edge set $E\Gamma\subseteq V\Gamma\times V\Gamma$. Simplicial means that $\Gamma$ has no double edges and that $\left(v,v\right)\notin E\Gamma$ for any $v\in V\Gamma$. We will always assume that $\Gamma$ is finite and undirected. For $v\in V\Gamma$ we denote by $\text{Link}\left(v\right)$ the set of all $w\in V\Gamma$ such that $\left(v,w\right)\in E\Gamma$. For $X \subseteq V\Gamma$ we shall write $\Link(X) := \cap_{v \in X} \Link(v)$ for the common link. Further, we define $\Link(\emptyset{}) := V \Gamma$ (all points). We may view $\Link(X)$ as a subgraph of $\Gamma$ by declaring the edge set of $\Link(X)$ to be $\{ (v,w) \in \Link(X) \times \Link(X) \mid (v,w) \in E\Gamma \}$, so that it inherits precisely the edges of $\Gamma$.

A clique in the graph $\Gamma$ is a subgraph $\Gamma_{0}\subseteq\Gamma$ in which every two vertices share an edge. We write $\Cliq$ or $\Cliq(\Gamma)$  for the set of {\bf cliques} in $\Gamma$ and $\Cliq(l)$ or $\Cliq(\Gamma,l)$ for the set of cliques with $l$ vertices. We will always assume that the empty graph is in $\Cliq$. $\text{Comm}\left(\Gamma_{0}\right)$ is the set of pairs $\left(\Gamma_{1},\Gamma_{2}\right)\in\text{Link}\left(\Gamma_{0}\right)\times\text{Link}\left(\Gamma_{0}\right)$ such that $\Gamma_{1},\Gamma_{2}\in\text{Cliq}$ and $\Gamma_{1}\cap\Gamma_{2}=\emptyset$.

 In this paper a {\bf word} refers to an expression $\boldv = v_1 v_2 \cdots v_n$ with $v_i \in V\Gamma$, i.e. a concatenation of elements in $V\Gamma$ which we call the {\bf letters}. Words will be denoted by bold face.  We say that two words are {\bf shuffle equivalent} (also known as {\bf II-equivalent}) if they are in the same equivalence class of the equivalence relation generated by
 \begin{itemize}
 \item $v_1 \cdots v_{i-1} v_i v_{i+1} v_{i+2} \cdots v_n \sim v_1 \cdots v_{i-1} v_{i+1} v_i v_{i+2} \cdots v_n$ if $(v_i, v_{i+1}) \in E\Gamma$.
 \end{itemize}
 We say that two words are equivalent, denoted by the symbol $\simeq$, if they are equivalent  through shuffle equivalence and the additional relation:
 \begin{itemize}
 \item $v_1 \cdots v_i v_{i+1} v_{i+2} \cdots v_n \sim v_1 \cdots v_i v_{i+2} \cdots v_n$ if $v_{i} = v_{i+1}$.
 \end{itemize}
 A word $v_1 \cdots v_n$ is called {\bf reduced} if whenever $v_i = v_j, i < j$ then there is $i < k < j$ such that $(v_i,v_k), (v_j, v_k) \not \in E\Gamma$.
  If $\boldv$ and $\boldw$ are reduced words and $\boldv \simeq \boldw$ then necessarily $\boldv$ and $\boldw$ are shuffle equivalent. For a word $\boldw$ we define its {\bf length} $\vert \boldw \vert$ as the number of letters in the shortest representative of $\boldw$ up to equivalence. We say that a word $\boldv$ {\bf starts with} $v \in V\Gamma$ if it is equivalent to a reduced word $v v_1 \cdots v_n, v_i \in V\Gamma$. Similarly we say that a word $\boldv$  {\bf ends with} $v \in V\Gamma$ if it is equivalent to a reduced word $v_1 \cdots v_n v, v_i \in V\Gamma$.

\subsection{Graph products of groups} \label{Graphs}
As before let $\Gamma$ be a simplicial graph. The following construction goes back to Green's thesis \cite{Green}.
For each $v\in V\Gamma$ let $G_{v}$ be a discrete group. Then the {\bf graph product group} $\ast_{v, \Gamma} G_{v}$ is the discrete group obtained from the free product of $G_{v}$, $v\in V\Gamma$ by adding additional relations $\left[s,t\right]=1$ for all $s\in G_{v}$, $t\in G_{w}$ with $\left(v,w\right)\in E\Gamma$. Special (extremal) cases of graph products are free products (graph with no edges) and Cartesian products (complete graphs). Right-angled Coxeter and right-angled Artin groups are natural examples of graph products.  A right-angled Coxeter group can be seen as a graph product where all groups $G_v, v \in V\Gamma$ are equal to $\mathbb{Z}_2$. A right-angled Artin group is a graph product with $G_v = \mathbb{Z}, v \in V\Gamma$.

\subsection{Graph products of operator algebras}\label{Sect=GraphProduct}
In \cite{CaspersFima} the concept of graph products has been translated to the operator algebraic setting. 
The construction is a generalization of free products by adding commutation relations depending on the underlying graph. We present a slightly different viewpoint of \cite{CaspersFima} by identifying  Hilbert spaces up to shuffle equivalence. This makes the notation much shorter and yields the same construction.

Let $\Gamma$ be a finite simplicial graph. For $v \in V\Gamma$ let $A_v$ be a unital C$^\ast$-algebra. Let $\varphi_v$ be a GNS-faithful state on $A_v$, meaning that the GNS-representation of $\varphi_v$ is faithful.  Set $A_v^\circ$ as the space of $a \in A_v$ with $\varphi_v(a) = 0$. For $a \in A_v$ set
\[
a^\circ = a - \varphi_v(a) 1 \in A_v^\circ.
\]
  Let $L^2(A_v^\circ, \varphi_v)$ be the closure of $A_v^\circ$ in the GNS-space $L^2(A_v, \varphi_v)$ of $\varphi_v$. We let $\Omega_v \in L^2(A_v, \varphi_v)$ be the unit of $A_v$.  For a reduced word $\boldv = v_1 \cdots v_n$ let
\[
\cH_\boldv = L^2(A_{v_1}^\circ, \varphi_{v_1}) \otimes \cdots \otimes  L^2(A_{v_n}^\circ, \varphi_{v_n} ).
\]
We have the convention $\mathcal{H}_\emptyset := \mathbb{C} \Omega$ where $\Omega$ is a unit vector called the {\bf vacuum vector}.
  If $\boldv = v_1 \cdots v_n$ and $\boldw = w_1 \cdots w_n$ are reduced equivalent (hence shuffle equivalent) words then $\cH_\boldv \simeq \cH_\boldw$ naturally by applying flip maps to the vectors dictated by the shuffle equivalence. More precisely, by \cite[Lemma 1.3]{CaspersFima} if $\boldv$ and $\boldw$ are equivalent reduced words, there exists a unique permutation $\sigma$  of the numbers $1, \ldots, n$ such that $v_{\sigma(i)} = w_i$ and such that if $i < j$ and $v_i = v_j$ then also $\sigma(i) < \sigma(j)$. Then there exists a unitary map $\cQ_{\boldv, \boldw}: \cH_\boldv \rightarrow \cH_\boldw$ which sends $\xi_1 \otimes \cdots \otimes \xi_n$ to   $\xi_{\sigma(1)} \otimes \cdots \otimes \xi_{\sigma(n)}$. From now on, we will omit  $\cQ_{\boldv, \boldw}$ in the notation and identify the spaces $\cH_\boldv$ and $\cH_\boldw$ through $\cQ_{\boldv, \boldw}$.  This  significantly simplifies our notation compared to \cite{CaspersFima} and it is directly verifyable that our constructions below agree with \cite{CaspersFima}.

   Let $I$ be a set of {\bf representatives} of all reduced words modulo shuffle equivalence. Set $\cH = \oplus_{\boldv \in I} \cH_{\boldv}$.  We represent each $A_v, v \in V\Gamma$ on $\cH$ as follows. Take $x \in A_v$.  Take $\xi_1 \otimes \cdots \otimes \xi_d \in \cH_\boldw$ with  $\boldw = w_1 \ldots w_d$ a reduced word. If $\boldw$ does not start with $v$ then we set
\[
x \cdot (\xi_1 \otimes \cdots \otimes \xi_d) = x^\circ \Omega_v \otimes \xi_1 \otimes \cdots \otimes \xi_d + \varphi_v(x)  \xi_1 \otimes \cdots \otimes \xi_d.
\]
If $\boldw$ starts with $v$ we may assume w.l.o.g. (by shuffling the letters if necessary and identifying corresponding Hilbert spaces as described above) that $w_1 = v$ and we set
\[
x \cdot (\xi_1 \otimes \cdots \otimes \xi_d) =
(x  \xi_1 - \langle x  \xi_1, \Omega_v \rangle \Omega_v ) \otimes \xi_2 \otimes \cdots \otimes \xi_d  +  \langle x  \xi_1, \Omega_v \rangle  \xi_2 \otimes \cdots \otimes \xi_d.
\]
This defines a faithful $\ast$-representation of $A_v$ on $\cH$. We set the {\bf (reduced) graph product C$^\ast$-algebra},
\[
(A, \varphi) := \ast_{v,\Gamma} (A_v, \varphi_v)
\]
as the C$^\ast$-algebra generated by all $A_v, v \in V\Gamma$ acting on $\cH$, with  faithful {\bf graph product state}
\[
\varphi(x) :=  \langle x \Omega, \Omega \rangle.
\]
If for every $v \in V\Gamma$, $A_v$ is a von Neumann algebra and $\varphi_v$ is moreover faithful then we define the {\bf graph product von Neumann algebra} as $(\ast_{v,\Gamma} (A_v, \varphi_v))''$. We will usually write
\[
L^2(A, \varphi) := \mathcal{H},
\]
and call this the (graph product) {\bf Fock space}.
The notation is justified as in \cite{CaspersFima} it is shown that $\mathcal{H}$ is the Hilbert space of the standard form of the von Neumann algebraic graph product.
An operator $a_1 \cdots a_n$ with $a_i \in A_{v_i}^\circ$ and $\boldv = v_1 \cdots v_n$ a  reduced word is called a {\bf reduced operator of type $\boldv$}. We refer to $n$ as the {\bf length} of the operator.  We define $P_v$ as the orthogonal projection of $\cH$ onto $\oplus_{\boldv \in I_v} \cH_\boldv$ where $I_v$ are representatives of all words that start with $v$ up to shuffle equivalence.
For $d \in \mathbb{N}_{\geq 0}$, let
\begin{equation}\label{Eqn=WordLengthProj}
\chi_d: A \rightarrow A: a_1 \cdots a_r \mapsto \delta(r = d) a_1 \cdots a_r,
\end{equation}
where $a_1 \cdots a_r$ is a reduced operator. So $\chi_d$ is the word length projection of length $d$.

\subsection{Coxeter groups} A {\bf Coxeter group} $W$ is a group that is freely generated by a set $S$ subject to relations $\left(st\right)^{m_{s,t}}=1$ for exponents $m_{s,t}\in\left\{ 1,2, \ldots, \infty\right\}$  where $m_{s,s}=1$, $m_{s,t}\geq2$ for all $s\neq t$ and $m_{s,t}=m_{t,s}$. The condition $m_{s,t}=\infty$ means that no relation of the form $\left(st\right)^{m}=1, m\in\mathbb{N}$ is imposed. The pair $\left(W,S\right)$ is called a {\bf Coxeter system}. The data of $\left(W,S\right)$ can be encoded in its {\bf Coxeter diagram}\footnote{We emphasize that we refer to this as the {\it diagram} of the Coxeter group, which is different from the graph defined in Section \ref{Sect=GraphCoxeter}.} with vertex set $S$ and edge set $\left\{ \left(s,t\right)\mid m_{s,t}\geq 3\right\}$  where every edge between two vertices $s,t\in S$ is labeled by the corresponding exponent $m_{s,t}$.

A {\bf special subgroup} of $W$ is one that is generated by a subset  $S_0$ of $S$ and which will be denoted by $W_{S_0}$. Special subgroups are also Coxeter groups with the same exponents as $W$ \cite[Theorem 4.1.6]{Da}. We briefly write $W_s = W_{\{ s \}}$ for $s \in S$ and note that by the relation $s^2 = e$ we have $W_s = \mathbb{Z}_2$, the group of two elements.  We call the Coxeter system $\left(W,S\right)$ \emph{irreducible} if its Coxeter diagram is connected. This is equivalent to $W$ not having a non-trivial decomposition into a direct product of special subgroups. Irreducible Coxeter groups fall into three classes.

\begin{definition}

Let $\left(W,S\right)$ be an irreducible Coxeter system.
\begin{itemize}
\item It is of {\bf spherical type} if it is locally finite, i.e. every finitely generated subgroup of $W$ is finite.
\item It is of {\bf affine type} if it is finitely generated, infinite and virtually abelian.
\item It is of {\bf non-affine type} if it is neither spherical nor affine.
\end{itemize}
\end{definition}

An irreducible Coxeter group is non-amenable if and only if it is of non-affine type, see \cite[Theorem 14.1.2 and Proposition 17.2.1]{Da}.

\subsection{Right-angled Coxeter groups as graph products}\label{Sect=GraphCoxeter} A Coxeter group (or Coxeter system) is called {\bf right-angled} if $m_{s,t} \in \{1,2, \infty \}$ for all $s,t \in S$. So we have that for $s \in S, s^2 =e$ and two letters $s,t \in S, s \not = t$ either commute (case $m_{s,t} = 2$) or they are free (case $m_{s,t} = \infty$). Suppose that $(W,S)$ is a right-angled Coxeter system. Set $\Gamma$ as the graph with vertices $V\Gamma = S$ and edge set $E\Gamma = \{ (s,t) \mid m_{s,t} = 2 \}$. Then we find an isomorphism of groups
\begin{equation} \label{Eqn=GraphIso}
W \simeq \ast_{s, \Gamma} W_{s} \simeq \ast_{s, \Gamma} \mathbb{Z}_2,
\end{equation}
sending $s_1 \cdots s_n \in W, s_i \in S$ to $s_1 \cdots s_n \in \ast_{s, \Gamma} W_{s}$ (i.e. our notation coincides). Indeed, this follows from the fact that the defining (universal) properties of a right-angled Coxeter group and a graph product over the special subgroups $W_s \simeq \mathbb{Z}_2$ are the same.

\subsection{Words in Coxeter groups}
Let $\boldw\in W$ with $\boldw=s_{1}\cdots s_{n}$ where $s_{1}, \ldots, s_{n}\in S$. We call this expression {\bf reduced} if it has minimal length, i.e. $n\leq m$ for every other representation $\boldw=t_{1}\cdots t_{m}, t_{1},\ldots,t_{m}\in S$. The set of letters that can occur in a reduced expression is independent of the choice of the reduced expression, \cite[Proposition 4.1.1]{Da}. By $\left|\boldw \right|:=n$ we define a {\bf word-length} on $W$. If $\left|\boldv^{-1} \boldw\right|=\left| \boldw \right|-\left| \boldv \right|$ (resp. $\left|\boldw \boldv^{-1}\right|=\left|\boldw\right|-\left| \boldv\right|$) for $\boldv\in W$ we say that {\bf $\boldw$ starts} (resp. {\bf ends}) {\bf with} $\boldv$. Note that in the right-angled case these definitions agree with the definitions for words and lengths on graphs  under the isomorphism \eqref{Eqn=GraphIso}.

Coxeter groups satisfy three important conditions on words in $S$ of which we will make use implicitly. We use the usual convention that $\widehat{s}$ means that $s$ is removed from an expression.

\begin{theorem}[\cite{Da}, Theorem 3.2.16 and Theorem 3.3.4]

Let $\left(W,S\right)$ be a Coxeter system, $\boldw=s_{1}\cdots s_{n}$ an expression for an element $\boldw\in W$ and $s,t\in S$. Then the following (equivalent) conditions hold:

\begin{itemize}
\item {\bf Deletion condition}: If $s_{1} \cdots s_{n}$ is not a reduced expression for $\boldw$, then there exist $i<j$ such that $s_{1} \cdots \widehat{s_{i}} \cdots \widehat{s_{j}} \cdots s_{n}$ is also an expression for $\boldw$.
\item  {\bf Exchange condition}: If $\boldw=s_{1} \cdots s_{n}$ is reduced, then either $\left|s\boldw\right|=n+1$ or there exists $1\leq i\leq n$ with $s\boldw=s_{1} \cdots \widehat{s_{i}} \cdots s_{n}$.
\item {\bf Folding condition}: If $\left|s\boldw\right|=\left|\boldw\right|+1$ and $\left|\boldw t\right|=\left| \boldw\right|+1$, then either $\left|s \boldw t\right|=\left|\boldw\right|+2$ or $\left|s \boldw t\right|=\left|\boldw \right|$.
\end{itemize}
\end{theorem}

 In the right-angled case, if we have cancellation of the form $s_{1} \cdots s_{n}=s_{1} \cdots \widehat{s_{i}} \cdots \widehat{s_{j}} \cdots s_{n}$ for $s_{1}, \ldots ,s_{n}\in S$, then $s_{i}=s_{j}$ and $s_{i}$ commutes with $s_{i+1}, \ldots, s_{j-1}$.


\subsection{Multi-parameter Hecke algebras} Von Neumann algebraic closures of Hecke algebras were first studied in \cite{Dymara1}, \cite{Dymara2}, see also \cite[Section 19]{Da}.  Note that we use a different normalization of the generators so that our notation coincides with Garncarek's \cite{Gar} (as well as \cite{CaspersAPDE} and \cite{CSW}).

Let $\left(W,S\right)$ be a Coxeter system and  let $\RWS$ (resp. $\CWS$ and $\OneWS$) be the set of tuples $q:=\left(q_{s}\right)_{s\in S}$ in $\mathbb{R}_{>0}^{ S }$ (resp. in $\mathbb{C}^{ S }$  and $\{ -1, 1\}^{ S }$) of positive real numbers (resp. complex numbers and numbers $-1$ or 1) with the property that $q_{s}=q_{t}$ whenever $s$ and $t$ are conjugate in $W$.

Take $q:=\left(q_{s}\right)_{s\in S} \in \RWS$.
 For every reduced expression $\boldw =s_{1}\cdots s_{n}$ of $\boldw\in W$ define
\[
q_{\boldw}:= q_{s_{1}} \cdots q_{s_{n}}, \qquad p_s(q):=q_{s}^{-\frac{1}{2}}\left(q_{s}-1\right).
\]
It follows from \cite[Proposition 19.1.1]{Da}\footnote{If $\widetilde{T}_{\boldw}^{\left(q\right)}$ are the generators from \cite[Proposition 19.1.1]{Da} then take normalized elements $T_{\boldw}^{\left(q\right)}= q_{\boldw}^{-1/2} \widetilde{T}_{\boldw}^{\left(q\right)}$ to get the relations \eqref{Eqn=Rel1} and \eqref{Eqn=Rel2}. This convention is consistent with \cite{CaspersAPDE}, \cite{CSW} and \cite{Gar}.} that there exists a unique $\ast$-algebra $\mathbb{C}_{q}\left[W\right]$ that is spanned  by a linear basis $\{ T_{\boldw}^{\left(q\right)}\mid \boldw \in W\}$ and such that for every $s\in S$ and $\boldw\in W$ we have,
 \begin{equation}\label{Eqn=Rel1}
T_{s}^{\left(q\right)}T_{\boldw}^{\left(q\right)}= \begin{cases}
T_{s\boldw}^{\left(q\right)} & \text{if }\left|s\boldw\right|>\left|\boldw\right|,\\
T_{s\boldw}^{\left(q\right)}+p_s(q)T_{\boldw}^{\left(q\right)} & \text{if }\left|s\boldw\right|<\left|\boldw\right|,
\end{cases}
\end{equation}
and
\begin{equation} \label{Eqn=Rel2}
\left(T_{\boldw}^{\left(q\right)}\right)^{\ast}=T_{\boldw^{-1}}^{\left(q\right)}.
\end{equation}
Denote by $\ell^2(W)$ the Hilbert space of square-summable functions on $W$ with canonical orthonormal basis $\left(\delta_{\mathbf w}\right)_{\mathbf{w}\in W}$.
 For every  $s\in S$ the element $T_{s}^{\left(q\right)}$ acts  boundedly on $\ell^2(W)$ by
\begin{equation}\label{Eqn=TOp}
\begin{array}{ll}
T_{s}^{\left(q\right)}\delta_{\boldw}:= \begin{cases}
\delta_{s\boldw} & \text{if }\left|s \boldw\right|>\left| \boldw\right|\\
\delta_{s\boldw}+p_s(q)\delta_{\boldw} & \text{if }\left|s \boldw\right|<\left| \boldw\right|
\end{cases}
\text{.}
\end{array}
\end{equation}
Moreover, this action extends to a faithful $\ast$-representation $\mathbb{C}_{q}\left[W\right] \rightarrow \cB(\ell^2(W))$.  We will usually identify  $\mathbb{C}_{q}\left[W\right]$ with its image in $\cB(\ell^2(W))$.

The unital $\ast$-algebra $\mathbb{C}_{q}\left[W\right]$  is called the (Iwahori) {\bf Hecke algebra}
 of $\left(W,S\right)$ with parameter $q$.   We set the {\bf reduced Hecke C$^{\ast}$-algebra} $C_{r,q}^{\ast}\left(W\right)$  as the norm closure of $\mathbb{C}_{q}\left[W\right]$ in $\cB\left(\ell^2(W)\right)$ using the representation \eqref{Eqn=TOp}. Then set the {\bf Hecke-von Neumann algebra} ${\mathcal N}_{q}\left(W\right) = C_{r,q}^{\ast}\left(W\right)''$.
  For $q_s=1, s \in S$ we get that $\mathbb{C}_{q}\left[W\right] = \mathbb{C}[W], C_{r,q}^{\ast}\left(W\right) = C_r^\ast(W)$ and ${\mathcal N}_{q}\left(W\right) = \mathcal{L}(W)$ are respectively the group algebra, reduced group C$^\ast$-algebra and group von Neumann algebra of $W$. Hence Hecke algebras are $q$-deformations of group (C$^\ast$- and von Neumann) algebras and the deformation (in principle) depends on $q$; we comment on this further in Section \ref{Sect=Iso}.
For every $q$ the vector state
\[
\tau(x) = \langle x \delta_{e}, \delta_{e} \rangle, \qquad x \in \cB(\ell^2(W)),
\]
  restricts to a {\bf tracial state} $\tau_q$ on $C_{r,q}^{\ast}\left(W\right)$ and ${\mathcal N}_{q}\left(W\right)$ with $\tau_q(T_{w}^{\left(q\right)} )=0$ for all $w\in W \setminus \{ e \}$.  Finally, define $P_{s}$, $s\in S$ to be the projection of $\ell^2(W)$ onto
\begin{equation}\label{Eqn=ProjectionSpace}
\overline{\text{span}\left\{ \delta_{\boldv}\mid \boldv\in W\text{ with }\left|s\boldv\right|<\left|\boldv\right|\right\} }\subseteq \ell^2(W)\text{.}
\end{equation}
Note that if the Coxeter group is right-angled then under the graph product identification \eqref{Eqn=GraphIso} this notation coincides with the notation of Section \ref{Sect=GraphProduct}.

  \subsection{Hecke C$^\ast$-algebras as graph products} Let $(W,S)$ be a right-angled Coxeter system with graph $\Gamma$ as defined in Section \ref{Sect=GraphCoxeter}.
  In \cite[Corollary 3.4]{CaspersAPDE}\footnote{The corollary is only stated for a single parameter $q_s = q$ but it also holds in the multi-parameter case with the same proof.} it is proved that for every $q = (q_{s})_{s \in S} \in \RWS$ there exists an isomorphism,
  \begin{equation}\label{Eqn=GraphHeckeIso}
    (C_{r,q}^{\ast}\left(W\right), \tau_q) \longrightarrow\!\!\!\!\!\!\!^\simeq \:\: \ast_{s, \Gamma}  (C_{r,q_s}^{\ast}\left(W_{s}\right), \tau_{q_s}).
  \end{equation}
  Moreover, the isomorphism is given by the canonical map $T_\boldv^{(q)} \rightarrow T_{s_1}^{(q_{s_1})} \cdots T_{s_d}^{(q_{s_d})}$ where $\boldv = s_1 \cdots s_d$ is reduced. For $q = 1$ this coincides with the corresponding result on Coxeter groups.

\section{A graph product Khintchine inequality} \label{A graph product Khintchine-Haagerup inequality}

The aim of this section is to prove a Khintchine inequality for general graph products. Such a Khintchine inequality estimates the operator norm of a reduced operator of length $d$   with the norm of certain Haagerup tensor products of column and row Hilbert spaces. This estimate of norms holds up to a bound that is polynomial in $d$. We make this more precise in the current section.

 \begin{rmk}
In the context of free products a Khintchine inequality was already obtained by Ricard and Xu in \cite[Section 2]{RicardXu}. If in the current paper we would only treat the case of free Coxeter systems and their Hecke deformations (i.e. $m(s,t) = \infty, s\not = t$) the results from \cite{RicardXu} would be sufficient. Here however, we want a more general theorem. One of the problems that arises while proving such a theorem is that the  analogue of \cite[Lemma 2.3]{RicardXu} in its form fails in a general graph product setting.
We remedy this problem by using maps which intertwine graph products with free products.
 \end{rmk}

Let us now prepare for the proof of the main theorem of this section. We fix   notation for both a graph product and a free product. As before, let $\Gamma$ be a finite simplicial graph  and let $I$ be a set of representatives of equivalence classes of reduced words with letters in  $V\Gamma$. Let $A_v, v \in V\Gamma$ be  unital C$^\ast$-algebras with GNS-faithful states $\varphi_v$.  Let
\[
(A, \varphi) := \ast_{v, \Gamma} (A_v, \varphi_v)
\]
 be its graph product with vacuum vector $\Omega$. We also set the free product (i.e. the graph product over $\Gamma$ with all edges removed)
 \[
 (A_f, \varphi_f) := \ast_{v} (A_v, \varphi_v),
 \]
 with vacuum vector $\Omega_f$.
 We shall view $A_v, v \in V\Gamma$ as a C$^\ast$-subalgebra of $A$ and $A_f$. We let $P_v$ be the graph product projection onto words that start with $v$ (as before) and we let $P_v^f$ be the free product projection onto words which start with $v$. Let $\Gamma_0 \in \Cliq(\Gamma, l)$ and let $\boldw = w_1 \cdots w_l \in I$ be the word consisting of all letters in $V\Gamma_0$. We write
 \begin{equation}\label{Eqn=PGammaF}
 P_{\Gamma_0}^f
 \end{equation}
  for the projection of $L^2(A_f, \varphi_f)$ onto the direct sum of all $\cH_{  \boldv}$ where $\boldv$ starts with $\boldw$.
 Recall that $A_v^\circ$ is the set of $a \in A_v$ with $\varphi_v(a) = 0$.

  View $A_v$ as a subalgebra of $A$ and then let $\Sigma_1 := {\rm span} \{ A_v^\circ \mid v \in V \Gamma\}$, and, for $d\in \mathbb{N}_{\geq 1}$,
   \[
   \Sigma_d := \{ a_1 \otimes \cdots \otimes a_d \mid a_i \in A_{v_i}^\circ \textrm{ and } v_1\cdots v_d \textrm{ reduced}  \} \subseteq \Sigma_1^{\otimes d},
\]
where the latter is the $d$-fold algebraic tensor product.

Our first aim is to show that reduced operators in $A$ of length $d$ can be decomposed as sums of creation operators, annihilation operators and diagonal operators in a sense to be made precise below. Crucial is that first the annihilation operators act, then the diagonal operators and then the creation operators. That is, we shall be looking for an analogue of the decomposition \cite[Fact 2.6]{RicardXu}.

\begin{rmk}\label{Rmk=Informal}
Definition \ref{Dfn=Sigma} below formally defines the following permutation.
Let $\boldv = v_1 \cdots v_d$ for $v_i \in V\Gamma$, be a reduced word. Let $d \in \mathbb{N}_{\geq 1}$, $0 \leq l \leq d$, $0 \leq k \leq d-l$, $\Gamma_0 \in \Cliq(\Gamma, l)$ and $(\Gamma_1, \Gamma_2) \in \Comm(\Gamma_0)$.  Then, if possible, we permute the letters of $\boldv$ through shuffle equivalence in the form:
\begin{equation}\label{Eqn=Form}
  (v_{\sigma(1)} \cdots v_{\sigma(k)} )  (v_{\sigma(k+1)} \cdots v_{\sigma(k+l)} ) (v_{\sigma(k+l+1)} \cdots   v_{\sigma(d)}) \simeq \overbrace{(\ast \cdots \ast V\Gamma_1)}^{k \textrm{  letters}} \overbrace{( \: V\Gamma_0 \:)}^{l \textrm{ letters}} \overbrace{(V \Gamma_2 \diamond \cdots \diamond)}^{d-l-k \textrm{ letters}},
\end{equation}
 where $\ast$ and $\diamond$ are the remaining letters and each of the 3 respective terms in between brackets are shuffle equivalent themselves. This means that in between the first brackets there is a word of length $k$ that ends on the clique $V\Gamma_1$, in between the second brackets there is the clique of length $l$ given by $V\Gamma_0$, and at the end there is a word of length $d - k - l$ that starts with the clique $V\Gamma_2$.

 Moreover, we want that $\ast \cdots \ast$ does not end on letters commuting with $V\Gamma_0$ and $V\Gamma_1$ and $\diamond \cdots \diamond$ does not start with letters commuting with  $V\Gamma_0$ and $V\Gamma_2$. This means that the cliques $\Gamma_1$ and $\Gamma_2$ are maximal for the property that a decomposition like \eqref{Eqn=Form} exists.

 If moreover we demand that $v_{\sigma(1)} \cdots v_{\sigma(k)}$,   $v_{\sigma(k+1)} \cdots v_{\sigma(k+l)}$ and  $v_{\sigma(k+l+1)} \cdots   v_{\sigma(d)}$ are in $I$, then there can be at most one such permutation coming from shuffle equivalences.

 Of course not for every $d, l, k, \Gamma_0, \Gamma_1, \Gamma_2$ and $\boldv$ this permutation exists since $\boldv$ cannot always be written in the from \eqref{Eqn=Form}.
\end{rmk}


\begin{dfn}\label{Dfn=Sigma}
Let $d \in \mathbb{N}_{\geq 1}$.
Suppose that $0 \leq l \leq d$, $0 \leq k \leq d-l$, $\Gamma_0 \in \Cliq(\Gamma, l)$ and $(\Gamma_1, \Gamma_2) \in \Comm(\Gamma_0)$. So if $l = 0$ we have that $\Gamma_0$ is the empty clique, and Condition \eqref{Item=Sigma2} below vanishes.
Take a reduced word $\boldv = v_1 \cdots v_d, v_i \in V\Gamma$. If existent,  define  $\sigma (= \sigma^\boldv_{l,k,\Gamma_0, \Gamma_1, \Gamma_2})$ as the permutation of indices $1, \ldots, d$   that satisfies:
 \begin{enumerate}
 \item\label{Item=Sigma1} $v_{1} \cdots v_d =  v_{\sigma(1)} \cdots  v_{\sigma(d)}$;
 \item\label{Item=Sigma2} $\{ v_{\sigma(k+1)}, \ldots, v_{\sigma(k+l)} \} = V\Gamma_0$;
 \item\label{Item=Sigma3} $\vert v_{\sigma(1)} \cdots v_{\sigma(k)} s\vert = k - 1$ whenever $s \in V\Gamma_1$;
 \item\label{Item=Sigma4} $\vert v_{\sigma(1)} \cdots v_{\sigma(k)} s\vert = k + 1$ whenever $s \in \Link(\Gamma_0) \backslash V \Gamma_1$;
 \item\label{Item=Sigma5} $\vert s v_{\sigma(k+l+1)} \cdots v_{\sigma(d)} \vert = d-k-l-1$ whenever $s \in V\Gamma_2$;
 \item\label{Item=Sigma6} $\vert s v_{\sigma(k+l+1)} \cdots v_{\sigma(d)} \vert = d-k-l + 1$ whenever $s \in \Link(\Gamma_0) \backslash V \Gamma_2$.
\end{enumerate}
We shall assume moreover that $v_{\sigma(1)} \cdots v_{\sigma(k)}$, $v_{\sigma(k+1)} \cdots v_{\sigma(k+l)}$ and $v_{\sigma(k+l+1)} \cdots v_{\sigma(d)}$ are in   $I$ (i.e. they are the representatives of their equivalence class) and if $v_{i} = v_{j}, i < j$ then $\sigma(i) < \sigma(j)$ so that  $\sigma$ comes from a shuffle equivalence. Then $\sigma$ is unique if it exists.
\end{dfn}

The permutation $\sigma$ of Definition \ref{Dfn=Sigma} does not necessarily exist. All expressions below in which a non-existing $\sigma$ occurs need to be interpreted as 0 and we shall recall this at the relevant  places.

\begin{example}
Consider the following graph:

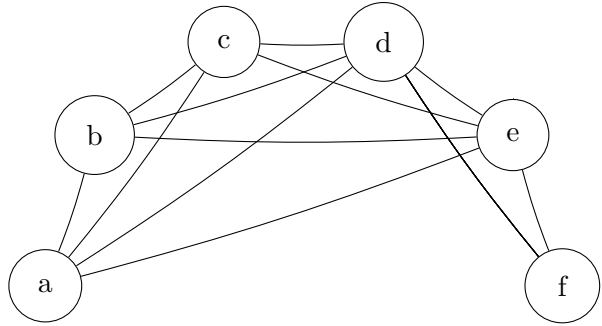
\begin{center}
\begin{tikzpicture}[transform shape]


      \node[draw,circle,inner sep=0.25cm] (N-1) at (180:3.4cm) {a};
      \node[draw,circle,inner sep=0.25cm] (N-2) at (144:3.4cm) {b};
      \node[draw,circle,inner sep=0.25cm] (N-3) at (108:3.4cm) {c};
      \node[draw,circle,inner sep=0.25cm] (N-4) at (72:3.4cm) {d};
      \node[draw,circle,inner sep=0.25cm] (N-5) at (36:3.4cm) {e};
      \node[draw,circle,inner sep=0.25cm] (N-6) at (0:3.4cm) {f};


  \foreach \number in {1,...,5}{
        \mycount=\number
        \advance\mycount by 1
  \foreach \numbera in {\the\mycount,...,5}{
    \path (N-\number) edge[-,bend right=3] (N-\numbera) ;
  }

    \path (N-4) edge[-,bend right=3] (N-6) ;

}
\end{tikzpicture}
\end{center}

\noindent This is the complete graph $K_5$ consisting of vertices $a, b, c, d, e$ together with an extra vertex $f$ that is connected only to $d$ and $e$.  Say that a word is in $I$ (i.e. is a representative) if it is minimal in alphabetical order amongst all equivalent words.  Now suppose that we have a reduced word:
\[
abcdef.
\]
\begin{itemize}
\item {\it Example 1.} Take $\Gamma_0 = \{ a, b, c \}$,   $\Gamma_1 = \{ d, e\}$ and  $\Gamma_2 = \emptyset$. Set $l = 3, k=2$.
Then $\sigma$ as in Definition \ref{Dfn=Sigma} exists and it is the permutation moving the word $abcdef$ to
\[
(de) (abc) f
\]
 (since every letter occurs uniquely it is clear what the permutation is). Indeed: $abc$ forms a clique; $de$ ends on $\Gamma_1$ and there is no other letter commuting with $\Gamma_1$ at the end of $de$; $f$ has no letters commuting with $abc$ at the start.
\item {\it Example 2.} Take $\Gamma_0 = \{ a, b \}$,   $\Gamma_1 = \{ d \}$ and  $\Gamma_2 = \{ c\}$. Set $l = 2, k=2$. Then $\sigma$ as in Definition \ref{Dfn=Sigma} does not exist. Indeed, by the choice of $k,l$ and $\Gamma_0$ if $\sigma$ exists there must be a word equivalent to $abcdef$ of the form ($\ast$ and $\diamond$ being undetermined letters):
\[
\ast \ast (ab) \diamond \diamond.
\]
By the choice of $\Gamma_1$ we see that  $\ast \ast$ ends on $d$ and there are no other letters in $\Link(\Gamma_0)$ at the end of $\ast \ast$. So we must have
\[
\ast d (ab) \diamond \diamond.
\]
However, there is no choice for the letter $\ast$ ($e$ is not allowed as it is in $\Link(\Gamma_0)$ and $f$ cannot be moved past $ab$).

 \end{itemize}
\end{example}

Now we find the following decomposition.

\begin{lem}\label{Lem=QaQDec}
Let $a_1 \cdots a_d$ be a reduced operator  of type $\boldv = v_1 \cdots v_d$  in the graph product $(A, \varphi)$.  Suppose that we have a non-zero expression of the form
\begin{equation}\label{Eqn=QExpression}
Q_{v_1}^{(1)} a_{1} Q_{v_1}^{(2)} \cdots Q_{v_d}^{(1)} a_{d} Q_{v_d}^{(2)},
\end{equation}
in $\cB(L^2(A, \varphi))$, where $Q_{v_i}^{(j)}$ equals either $P_{v_i}$ or $P_{v_i}^\perp$. Then, for some permutation $\alpha$ of the indices $1,\ldots, d$ coming from a shuffle equivalence and for some $0 \leq r \leq d$ and $0 \leq m \leq d-r$ we have that  \eqref{Eqn=QExpression} equals
\begin{equation}\label{Eqn=ShuffleQaQ}
\begin{split}
&  (P_{v_{ \alpha(1) } } a_{ \alpha(1) } P_{v_{ \alpha(1) }}^\perp )  \cdots   ( P_{v_{\alpha(r)  }} a_{\alpha(r)} P_{  v_{ \alpha(r) } }^\perp )  (P_{ v_{\alpha(r+1) } }   a_{ \alpha(r+1) }  P_{v_{\alpha(r+1) } }  )\cdots (P_{v_{\alpha(m)}}  a_{\alpha(m)} P_{v_{\alpha(m)}}  )\\
 & \quad \times \quad ( P_{v_{\alpha(m+1)} }^\perp a_{ \alpha(m+1) } P_{v_{ \alpha(m+1) } }) \cdots ( P_{v_{\alpha(d) } }^\perp a_{ \alpha(d)  } P_{ v_{\alpha(d) } }).
\end{split}
\end{equation}
Further, for a non-zero expression of the form \eqref{Eqn=ShuffleQaQ} we have that $v_{\alpha(r+1)} \cdots v_{\alpha(m)}$ is in a clique.
\end{lem}
\begin{proof}
We prove this in a series of claims. To avoid cumbersome notation we shall not write the permutation of the shuffle equivalences in the proof.

\vspace{0.3cm}

\noindent {\bf Claim 1.} The expression \eqref{Eqn=QExpression} is up to shuffle equivalence equal to:
\begin{equation}\label{Eqn=QExpression1}
Q_{v_1}^{(1)} a_{1} Q_{v_1}^{(2)} \cdots  Q_{v_m}^{(1)} a_m Q_{v_m}^{(2)} (P_{v_{m+1} }^\perp  a_{m+1} P_{v_{m+1} }) \cdots (P_{v_{d} }^\perp a_{d} P_{v_{d} }).
\end{equation}
Moreover, the tail of annihilation operators is maximal in the sense that if for some $i\leq m$ we have $Q_{v_i}^{(2)} = P_{v_i}$ then  $Q_{v_i}^{(1)} = P_{v_i}$.

\vspace{0.3cm}

\noindent {\it Proof of Claim 1.}  Suppose that we are given an expression as in \eqref{Eqn=QExpression1}. Suppose that for some $i<m$ we have $Q_{v_i}^{(1)} = P_{v_i}^\perp, Q_{v_i}^{(2)} = P_{v_i}$.   Then we need to show that $v_i$ commutes with $v_{i+1}\cdots v_{m}$. To do so we may suppose the index $i$ was chosen maximally. Suppose that $v_i$ and $v_{i+1}\cdots v_{m}$ do not commute and let $v_k$ be the first letter in $v_{i+1}\cdots v_{m}$ that does not commute with $v_i$. Our choice of $i$ yields that $Q_{v_k}^{(1)} = P_{v_k}$. Indeed if $Q_{v_k}^{(1)}$ were to be $P_{v_k}^\perp$ then  \eqref{Eqn=QExpression1} is 0 in case $Q_{v_k}^{(2)} = P_{v_k}^\perp$ and  in case $Q_{v_k}^{(2)} = P_{v_k}$ this would contradict maximality of $i$. But then \eqref{Eqn=QExpression1} contains a factor $P_{v_i} P_{v_k} = 0$ which means that \eqref{Eqn=QExpression1} would be zero which in turn is a contradiction.

\vspace{0.3cm}

\noindent {\bf Claim 2.} The expression \eqref{Eqn=QExpression} is up to shuffle equivalence equal to:
\begin{equation}\label{Eqn=QExpression21}
\begin{split}
& Q_{v_1}^{(1)} a_{1} Q_{v_1}^{(2)} \cdots  Q_{v_r}^{(1)} a_{r} Q_{v_r}^{(2)}  (P_{v_{r+1}}  a_{r+1} P_{v_{r+1}}  )\cdots (P_{v_m}  a_{m} P_{v_m}  )\\
 & \quad \times \quad (P_{v_{m+1} }^\perp a_{m+1} P_{v_{m+1} }) \cdots ( P_{v_{d} }^\perp a_{d} P_{v_{d} }).
\end{split}
\end{equation}
Moreover, the tail of annihilation and diagonal operators is maximal in the sense that if for some $i\leq r$ we have $Q_{v_i}^{(1)} = P_{v_i}$ then  $Q_{v_i}^{(2)} = P_{v_i}^\perp$.

\vspace{0.3cm}

\noindent {\it Proof of Claim 2.}  Suppose that we are given a non-zero expression as in \eqref{Eqn=QExpression21}. Suppose that for some $i<r$ we have $Q_{v_i}^{(1)} = P_{v_i}, Q_{v_i}^{(2)} = P_{v_i}$.   Then we need to show that $v_i$ commutes with $v_{i+1}\cdots v_{r}$. To do so we may suppose the index $i < r$ was chosen maximally. Suppose that $v_i$ and $v_{i+1}\cdots v_{r}$ do not commute and let $v_k$ be the first letter in $v_{i+1}\cdots v_{r}$ that does not commute with $v_i$. We claim that our choice of $i$ yields that $Q_{v_k}^{(2)} = P_{v_k}^\perp$. Indeed, suppose that $Q_{v_k}^{(2)} = P_{v_k}$. Then if $Q_{v_k}^{(1)} = P_{v_k}$ this contradicts maximality of $i$  and if $Q_{v_k}^{(1)} = P_{v_k}^\perp$ it would contradict Claim 1. From $Q_{v_k}^{(2)} = P_{v_k}^\perp$ we find that $Q_{v_k}^{(1)} = P_{v_k}$  since if $Q_{v_k}^{(1)} = P_{v_k}^\perp$ then $P_{v_k}^\perp a_k P_{v_k}^\perp = 0$.   But then \eqref{Eqn=QExpression21} contains the factor $P_{v_i} P_{v_k} = 0$ with $v_i$ and $v_k$ non-commuting,  which means that \eqref{Eqn=QExpression21} would be zero. As this is a contradiction the claim follows.

\vspace{0.3cm}

\noindent {\bf Claim 3.} The expression \eqref{Eqn=QExpression} is up to shuffle equivalence equal to:
\begin{equation}\label{Eqn=QExpression22}
\begin{split}
&  (P_{v_{1} } a_{1} P_{v_1}^\perp )  \cdots   (P_{v_{r} } a_{r} P_{v_r}^\perp )  (P_{v_{r+1}}  a_{ r+1 } P_{v_{r+1}}  )\cdots (P_{v_m}  a_{m} P_{v_m}  )\\
 & \quad \times \quad ( P_{v_{m+1} }^\perp a_{ m+1 } P_{v_{m+1} }) \cdots ( P_{v_{d} }^\perp a_{ d } P_{v_{d} }).
\end{split}
\end{equation}
Moreover $v_{r+1} \cdots v_m$ forms a clique.

\vspace{0.3cm}

\noindent {\it Proof of Claim 3.} This is obvious now from Claim 2 and the fact that $P_{v_{i} }^\perp a_{v_i} P_{v_i}^\perp = 0$. As $P_{v_i} P_{v_j}$ is non-zero only if $v_i$ and $v_j$ commute we must have that $v_{r+1} \cdots v_m$ forms a clique.

\vspace{0.3cm}

\noindent {\it Proof of the lemma.} We may now directly conclude the lemma from Claim 3. The permutation $\alpha$ is then the composition of the shuffle equivalences coming from claims 1, 2 and 3.

\end{proof}

\begin{prop}\label{Prop=TProjectionDec}
Let $a_1 \cdots a_d$ be a reduced operator  of type $\boldv = v_1 \cdots v_d$  in the graph product $(A, \varphi)$.
We have the following equality of  operators in  $\cB(L^2(A, \varphi))$
\begin{equation}\label{Eqn=TProjectionDec}
\begin{split}
a_1  \cdots  a_d = & \sum_{l=0}^{d} \sum_{k=0}^{d-l} \sum_{\Gamma_0 \in \Cliq(\Gamma, l)} \sum_{(\Gamma_1, \Gamma_2) \in \Comm(\Gamma_0)}
 ( P_{v_{\sigma(1)}} a_{ \sigma(1) }  P_{v_{\sigma(1)}}^\perp)   \cdots   (P_{v_{\sigma(k)}} a_{\sigma(k)} P_{v_{\sigma(k)}}^\perp)  \\
 & \times  (P_{v_{\sigma(k+1) } }   a_{ \sigma(k+1) }  P_{v_{\sigma(k+1) } } ) \cdots   (  P_{v_{\sigma(k+l)}}  a_{\sigma(k+l)}  P_{v_{\sigma(k+l)}}) \\
& \times  (P_{v_{\sigma(k+l+1) } }^\perp a_{ \sigma(k+l+1) }  P_{v_{\sigma(k+l+1) } } ) \cdots   (  P_{v_{\sigma(d)}}^\perp  a_{\sigma(d)}  P_{v_{\sigma(d)}}),
 \end{split}
\end{equation}
where $\sigma$ (changing over the summation) is as in Definition \ref{Dfn=Sigma}. If such $\sigma$ does not exist then the summand is understood as 0.
\end{prop}
\begin{proof}
We first note that we may decompose,
\begin{equation}\label{Eqn=ProjectionDecomposition}
a_1  \cdots  a_d  =
(P_{v_1} + P_{v_1}^\perp) a_{1}  (P_{v_1} + P_{v_1}^\perp) \cdots  (P_{v_d} + P_{v_d}^\perp) a_{d} (P_{v_d} + P_{v_d}^\perp).
\end{equation}
What we showed in Lemma \ref{Lem=QaQDec} and  \eqref{Eqn=ProjectionDecomposition} is that the product $a_1 \cdots a_d$ decomposes as a sum of operators of the form \eqref{Eqn=ShuffleQaQ} (where $\alpha$ depends on the summand). The proof is finished by arguing that the
 summation in \eqref{Eqn=TProjectionDec} runs exactly over all these summands.

  It is clear that each (non-zero) summand in \eqref{Eqn=TProjectionDec} is an expression of the form \eqref{Eqn=ShuffleQaQ} with $r=k$ and $l = m - r$. Conversely, take an expression of the form \eqref{Eqn=ShuffleQaQ}, then the letters $v_{\alpha(r+1)} \cdots v_{\alpha(m)}$ form a clique $\Gamma_0$ by Lemma \ref{Lem=QaQDec}. Set $l = \# V\Gamma_0$ and $k = r$.  Now in $v_{\alpha(1)}, \ldots, v_{\alpha(r)}$ there may be letters at the end that commute with $\Gamma_0$. Let $\Gamma_1$ be the clique of letters that appear at the end of $v_{\alpha(1)}, \ldots, v_{\alpha(r)}$ that commute with $\Gamma_0$ that is maximal in the following sense: there are no letters $s$  appearing at the end of $v_{\alpha(1)} \cdots v_{\alpha(r)}$ that commute with $\Gamma_0$ and $\Gamma_1$. Such a clique is unique since if both $\Gamma_1$ and $\Gamma_1'$ would be such cliques, then so is $\Gamma_1 \cup \Gamma_1'$ and hence $\Gamma_1 = \Gamma_1'$ by maximality. Similarly we may let $\Gamma_2$ be a clique of letters that appear at the start of $v_{\alpha(m+1)} \cdots v_{\alpha(d)}$ that commute with $\Gamma_0$ and that is maximal. Then for this choice of $\Gamma_0, \Gamma_1, \Gamma_2$ we have that  $\sigma$ satisfying \eqref{Item=Sigma1} - \eqref{Item=Sigma6} exists and it is moreover the only choice for which it exists. This shows that each non-zero expression  \eqref{Eqn=ShuffleQaQ} occurs exactly once in the summation \eqref{Eqn=TProjectionDec}.
\end{proof}

In order  to prove our Khintchine inequality we introduce the necessary notation.
 We define as in \cite[Section 2]{RicardXu} the following subspaces of $\cB(L^2(A_f, \varphi_f))$,
 \[
    L_1 = {\rm span} \{   P_v^f a_v P_v^{f\:\perp} \mid v \in V\Gamma, a \in A_v \}, \qquad K_1 = L_1^\ast.
 \]
 It is proved in \cite[Lemma 2.3]{RicardXu} that
 \begin{equation}\label{Eqn=Column}
 L_1 \simeq \left( \oplus_{v \in V\Gamma}  L^2(A_v^\circ, \varphi_v) \right)_{ \rmC }, \qquad  K_1 \simeq \left( \oplus_{v \in V\Gamma}  L^2(A_v^\circ, \varphi_v) \right)_{\rmR}
  \end{equation}
  completely isometrically, where the subscript $\rmC$ (resp. $\rmR$) denotes the column (resp. row) Hilbert space structure. In particular, if each $A_v^\circ$ is one dimensional (as is the case for right-angled Hecke algebras) we have that $L_1$ (resp. $K_1$) is completely isometrically isomorphic to the column Hilbert space  $\rmC_{\# V\Gamma}$ (resp. row Hilbert space $\rmR_{\# V \Gamma}$).  We set $k$-fold Haagerup tensor products,
 \[
 L_k = L_1^{\hotimes k}, \qquad K_k = K_1^{\hotimes k}.
 \]

Fix $\Gamma_0 \in \Cliq(l)$. For $a_1, \ldots,  a_l$ with $a_i \in A_{v_i}$ and  $\boldv = v_1 \cdots v_l$ a reduced word (for the graph product) in $I$ consisting precisely of all letters of $\Gamma_0$ we define an element of $\cB(L^2(A_f, \varphi_f))$ by setting for $r > l$,
 \begin{equation}\label{Eqn=AGamma}
 \Diag(a_1, \ldots ,  a_l):  b_1 \cdots b_l b_{l+1} \cdots  b_r \Omega_f \mapsto  \overbrace{a_1 b_1}^\circ \cdots  \overbrace{a_l b_l}^\circ b_{l+1} \cdots  b_r \Omega_f,
 \end{equation}
where $b_i \in A_{w_i}^\circ$ with $w_{i} \not = w_{i+1}$ (so $b_1 \cdots b_r$ is a reduced word in the free product).   If $r < l$ then the image in \eqref{Eqn=AGamma} is 0.

 \begin{lem}
 The operator defined in \eqref{Eqn=AGamma} is bounded.
 \end{lem}
 \begin{proof}
Let $r \in \mathbb{N}_{\geq 1}$. Fix a  word $\boldw = w_1 \cdots w_r$ with $w_i \not = w_{i+1}$ for all $1\leq i<r$. Then,
 \[
 L^2(A_{w_1}^\circ, \varphi_{w_1}) \otimes \cdots \otimes  L^2(A_{w_r}^\circ, \varphi_{w_r})
 \]
 is an invariant subspace for the action \eqref{Eqn=AGamma}. Moreover, note that \eqref{Eqn=AGamma} is for $r \geq l$ just the tensor product operator
 \begin{equation}\label{Eqn=DiagExplain}
 P_{w_1} a_1 P_{w_1} \otimes \cdots \otimes P_{w_l} a_l P_{w_l} \otimes 1^{\otimes r-l},
   \end{equation}
   where $P_{w_i} a_i P_{w_i}$ acts on $L^2(A_{w_i}^\circ, \varphi_{w_i})$. Clearly this operator is bounded.
 \end{proof}

Set the diagonal space
\begin{equation}\label{Eqn=AGammaNot}
A_{\Gamma_0} \subseteq \cB(L^2(A_f, \varphi_f))
\end{equation}
 to be the linear span of all operators of the form \eqref{Eqn=AGamma}. $A_{\Gamma_0}$ inherits the operator space structure of $A_f$.
 Set for $d \in \mathbb{N}_{\geq 1}$,
 \begin{equation}\label{Eqn=Xd2}
X_d =  \bigoplus_{l=0}^{d}   \bigoplus_{k=0}^{d-l}   \bigoplus_{\Gamma_0 \in {\rm Cliq}(\Gamma, l)} \bigoplus_{(\Gamma_1, \Gamma_2) \in \Comm(\Gamma_0)}L_{k} \hotimes A_{\Gamma_0} \hotimes K_{d-k-l}.
\end{equation}
 In case $A_v^\circ, v \in V\Gamma$ are all 1-dimensional, the space $X_d$ can also be understood in terms of bounded operators on a Hilbert space. Indeed, the remarks after \eqref{Eqn=Column}   and  \cite[Proposition 3.5]{BlecherPaulsen} give the first two  completely isometric isomorphisms of:
 \begin{equation}\label{Eqn=IsoIso}
 \begin{split}
 L_{k} \hotimes A_{\Gamma_0} \hotimes K_{d-k-l} \simeq  & \rmC_{k} \hotimes A_{\Gamma_0} \hotimes \rmR_{d-k-l} \\
  \simeq & M_{k, d-k-l}(A_{\Gamma_0}) \simeq  M_{k, d-k-l}(\mathbb{C}) \minotimes A_{\Gamma_0}.
 \end{split}
 \end{equation}
 The third completely isometric isomorphism of \eqref{Eqn=IsoIso} holds by definition of the operator space structure on $A_{\Gamma_0}$ as part of the C$^\ast$-algebra $A_f$.
 Next consider the following embedding
 \[
 j_d: \Sigma_d \rightarrow X_d,
 \]
  where the image of $a_1 \otimes \cdots \otimes a_d$ is given as follows: Consider a summand of $X_d$ indexed by  $(l,k,\Gamma_0, \Gamma_1, \Gamma_2)$ with $0 \leq l \leq d, 0 \leq k \leq d-l$ and $\Gamma_0 \in \Cliq(\Gamma,l), (\Gamma_1, \Gamma_2) \in \Comm(\Gamma_0)$. Then the restriction of the image of $j_d$ to this summand is given by
 \begin{equation}\label{Eqn=JdDefinition}
 \begin{split}
 j_d(a_1 \otimes \cdots \otimes a_d)\vert_{X_d}
 = & ( P_{v_{\sigma(1)}}^f a_{ \sigma(1) }  P_{v_{\sigma(1)}}^{f\:\perp}) \otimes   \cdots \otimes    (P_{v_{\sigma(k)}}^{f} a_{ \sigma(k) }  P_{v_{\sigma(k)}}^{f \: \perp} )  \\
 & \otimes   \Diag(a_{ \sigma(k+1)  } , \ldots  ,  a_{ \sigma(k+l) }    )    \\
& \otimes  (P_{v_{\sigma(k+l+1) } }^{f \: \perp} a_{  \sigma(k+l+1) }  P_{v_{\sigma(k+l+1) } }^f ) \otimes \cdots \otimes    (  P_{v_{\sigma(d)}}^{f\: \perp}  a_{ \sigma(d)}  P_{v_{\sigma(d)}}^f ),
 \end{split}
 \end{equation}
 with $\sigma$  given by Definition \ref{Dfn=Sigma}; if such $\sigma$ is non-existent then the image of  $j_d(a_1 \otimes \cdots \otimes a_d)$ in  the summand of $X_d$ corresponding to $(l,k,\Gamma_0, \Gamma_1, \Gamma_2)$ is 0.
Let
\[
\pi_d^f: X_d \rightarrow \mathcal{B}(L^2(A_f, \varphi_f))
\]
be the direct sums of product maps.
For a 5-tuple $(l,k,\Gamma_0, \Gamma_1, \Gamma_2)$ as above, let
\begin{equation}\label{Eqn=PiDs}
\pi_{d, l, k,\Gamma_0, \Gamma_1, \Gamma_2}^f: L_{k} \hotimes A_{\Gamma_0} \hotimes K_{d-k-l} \rightarrow \mathcal{B}( L^2(A_f),\varphi_f )
\end{equation}
be the product map $\pi_d^f$ restricted to the corresponding summand of $X_d$.   This map is completely bounded as follows from the definition of the Haagerup tensor product.
Consequently, $\pi_d^f: X_d \rightarrow \mathcal{B}(L^2(A_f, \varphi_f))$ is completely bounded by the number of summands of $X_d$, i.e.
\[
\Vert \pi_d^f \Vert_{cb} \leq (\# \Cliq(\Gamma) )^3 d.
\]

\vspace{0.3cm}

\vspace{0.3cm}

\noindent {\it Definition of two partial isometries.} Given a 5-tuple  $(l,k, \Gamma_0, \Gamma_1, \Gamma_2)$ as in the previous paragraph, we define two partial isometries.
\begin{itemize}
\item We define a partial isometry,
\begin{equation}
\begin{split}
\mathcal{Q}_{l,k,\Gamma_0,   \Gamma_1, \Gamma_2}: & L^2(A, \varphi) \rightarrow L^2(A_f, \varphi_f),
\end{split}
\end{equation}
as follows. Consider a reduced operator $b_{1} \cdots b_n \in A$ of type $\boldw$. We need to define a permutation $\sigma_{\cQ} = \sigma^\boldw_{\cQ, l,k, \Gamma_0, \Gamma_1, \Gamma_2}$   coming from a shuffle equivalence satisfying \eqref{Item=Sigma1} -- \eqref{Item=Sigma4} of Definition \ref{Dfn=Sigma} and the additional relation that $\vert s w_{\sigma_{\cQ}(k+l+1)} \cdots w_{\sigma_{\cQ}(n)} \vert = n-k-l+1$ whenever $s \in V \Gamma_2$.  Moreover we assume that this $\sigma_{\cQ}$ is chosen such that each of the expressions $w_{\sigma_{\cQ}(k)} $ $\cdots$ $w_{\sigma_{\cQ}(1)}$ (decreasing indices), $w_{\sigma_{\cQ}(k+1)}  \cdots  w_{\sigma_{\cQ}(k+l)}$   and $w_{\sigma_{\cQ}(k+l+1)} \cdots w_{\sigma_{\cQ}(n)}$   are in $I$. If $\sigma_{\cQ}$ exists it is unique and we set
 \[
 \mathcal{Q}_{l,k,\Gamma_0,   \Gamma_1, \Gamma_2}  b_{1} \cdots b_n\Omega  = b_{ \sigma_{\cQ}(1) } \cdots b_{\sigma_{\cQ}(n)}  \Omega_f.
 \]
  If $\sigma_{\cQ}$ does not exist we set  $\mathcal{Q}_{l,k ,\Gamma_0,   \Gamma_1, \Gamma_2} b_{1} \cdots b_n \Omega = 0$.

\item We define the partial isometry,
\begin{equation}
\begin{split}
\mathcal{R}_{l, k,\Gamma_0,   \Gamma_1}: & L^2(A, \varphi) \rightarrow L^2(A_f, \varphi_f),
\end{split}
\end{equation}
as follows.  Consider a reduced operator $b_{1} \cdots b_n \in A$ of type $\boldw$.  We need to define a permutation  $\sigma_{\cR} = \sigma^\boldw_{\cR, l,k, \Gamma_0, \Gamma_1}$  coming from a shuffle equivalence  satisfying \eqref{Item=Sigma1} -- \eqref{Item=Sigma4} of Definition \ref{Dfn=Sigma}.
 Moreover we assume that this $\sigma_{\cR}$ is chosen such that each of the expressions $w_{\sigma_{\cR}(1)} $ $\cdots$ $w_{\sigma_{\cR}(k)}$, $w_{\sigma_{\cR}(k+1)} $$\cdots $$w_{\sigma_{\cR}(k+l)}$ and $w_{\sigma_{\cR}(k+l+1)} \cdots w_{\sigma_{\cR}(n)}$ are in $I$. If $\sigma_{\cR}$ exists then it is unique  and we set
 \[
 \mathcal{R}_{l, k,\Gamma_0,   \Gamma_1}  b_{1} \cdots b_n\Omega  = b_{ \sigma_{\cR}(1) } \cdots b_{\sigma_{\cR}(n)}  \Omega_f.
 \]
  If $\sigma_{\cR}$ does not exist we set  $\mathcal{R}_{l, k,\Gamma_0,   \Gamma_1} b_{1} \cdots b_n \Omega = 0$.
 \end{itemize}

\vspace{0.3cm}

\noindent The maps $\mathcal{Q}_{l, k,\Gamma_0,   \Gamma_1, \Gamma_2}$ and $\mathcal{R}_{l, k,\Gamma_0,   \Gamma_1}$ preserve orthogonality and inner products and are therefore partial isometries.

\begin{prop}\label{Prop=GeneralKin}
Let $x = a_{1} \otimes \cdots \otimes a_d  \in \Sigma_d$ be of type $\boldv = v_1 \cdots v_d$ and let $x_{d,l,k,\Gamma_0, \Gamma_1, \Gamma_2}$ with  $0 \leq l \leq d$, $0 \leq k \leq d-l$,  $\Gamma_0 \in \Cliq(\Gamma,l)$ and $(\Gamma_1, \Gamma_2) \in \Comm(\Gamma_0)$ be the corresponding summands of $j_d(x)$ in $X_d$ as in \eqref{Eqn=JdDefinition}.   We have
 \begin{equation}\label{Eqn=TediousClaim}
 \begin{split}
& \mathcal{R}_{l,k,\Gamma_0, \Gamma_1}^\ast  \pi_{d, l,k,\Gamma_0, \Gamma_1, \Gamma_2}^f(  x_{d,l,k,\Gamma_0, \Gamma_1, \Gamma_2}  )
 \mathcal{Q}_{l, d-l-k,  \Gamma_0, \Gamma_2, \Gamma_1} \\
 =&
  (  P_{v_{\sigma(1)} }  a_{\sigma(1) }  P_{v_{\sigma(1)} }^\perp )   \cdots   ( P_{v_{\sigma(k)}}    a_{\sigma(k)}   P_{v_{\sigma(k)}}^\perp ) \\
   & \quad  \times \quad   ( P_{v_{\sigma(k+1)}}    a_{\sigma(k+1)}   P_{v_{\sigma(k+1)}} )  \cdots ( P_{ v_{\sigma(k+l)} }    a_{\sigma(k+l)}   P_{v_{\sigma(k+l)}} ) \\
   &   \quad \times  \quad   (P_{v_{\sigma(k+l+1)}}^\perp   a_{\sigma(k+l+1)}  P_{v_{\sigma(k+l+1)}} )  \cdots   (P_{v_{\sigma(d)}}^\perp a_{\sigma(d)} P_{v_{\sigma(d)}}),
 \end{split}
 \end{equation}
 where $\sigma$ is defined as in (1) -- (6) of Definition \ref{Dfn=Sigma} and the right hand side should be understood as 0 otherwise.
\end{prop}
\begin{proof}
  Note that both sides of \eqref{Eqn=TediousClaim} equal 0 if a $\sigma$ as in the statement  of the proposition does not exist, c.f. the definition of $j_d$. So from now on we  assume that  $\sigma$ exists and that the right hand side of \eqref{Eqn=TediousClaim} is non-zero for some elementary tensor product $a_{1} \otimes \cdots \otimes a_{d} \in \Sigma_d$ of type $v_1 \cdots v_d$.

  We argue first that without loss of generality we may assume that the permutation $\sigma$ on the right hand side of \eqref{Eqn=TediousClaim} is trivial. Indeed, if $\sigma$ is non-trivial then we may work with the element $x' = a_{\sigma(1)} \otimes \cdots \otimes a_{\sigma(d)} \in \Sigma_d$ instead of $x$. Then note that the left hand side   of \eqref{Eqn=TediousClaim} for $x$ and $x'$ are the same. Similarly the right hand side of \eqref{Eqn=TediousClaim} is the same for $x$ and $x'$.

So assume that $\sigma$ is trivial.   Now take a reduced operator $b_{1}   \cdots   b_{r}, r \geq 0$  of type $\boldw = w_1 \cdots w_r$. We first prove the proposition for the case that the permutation $\sigma_{\cQ} := \sigma_{\cQ, l, d-l-k, \Gamma_0, \Gamma_2, \Gamma_1}^\boldw$ exists. In that case  set,
\[
\cQ_{l, d-l-k , \Gamma_0, \Gamma_2, \Gamma_1}( b_1  \cdots b_r \Omega ) = b_{1}' \cdots b_r' \Omega_f,
\]
with $b_i' = b_{\sigma_{\cQ}(i)}, w_i' = w_{ \sigma_{\cQ}(i) }$. Further,
\[
w_1' \cdots w_r' \simeq \overbrace{(\ast \cdots \ast V\Gamma_2)}^{d-l-k} (V\Gamma_0) (\diamond \cdots \diamond),
\]
 where $(\diamond \cdots \diamond)$ has no letters in $V\Gamma_1$ at the start.
 We have by definition of $\pi^f_{d,l,k,\Gamma_0, \Gamma_1, \Gamma_2}$,
\[
\begin{split}
& \pi^f_{d,l,k,\Gamma_0, \Gamma_1, \Gamma_2} (x_{d,l,k, \Gamma_0, \Gamma_1, \Gamma_2} ) \\
= &
( P_{v_{1}}^f a_{1 }  P_{v_{1}}^{f \:\perp} )   \cdots   (P_{v_{k}}^f a_{k}  P_{v_{k}}^{f \: \perp} )
     \Diag(    a_{ k+1 },   \cdots,       a_{ k+l}   )
  ( P_{v_{k+l+1}}^{f \: \perp}  a_{ k+l+1}  P_{v_{k+l+1}}^{f} ) \cdots   (  P_{v_{ d } }^{f \: \perp} a_{ d }  P_{v_{ d } }^{f} ).
\end{split}
 \]
Then, for the left hand side of \eqref{Eqn=TediousClaim},
\begin{equation}\label{Eqn=FirstStuff}
\begin{split}
&  \pi^f_{d,l,k,\Gamma_0, \Gamma_1, \Gamma_2} (x_{d,l,k, \Gamma_0, \Gamma_1, \Gamma_2} )\cQ_{ l,d-l-k, \Gamma_0, \Gamma_2, \Gamma_1}( b_1 \cdots b_r \Omega)\\
= & \langle a_{d} b_{1}'\Omega_f,  \Omega_f \rangle \cdots \langle a_{k+l+1} b_{d-k-l}' \Omega_f,  \Omega_f \rangle
 a_1 \cdots a_k    \overbrace{(a_{k+1} b_{d-k-l+1}')}^\circ \cdots \overbrace{(a_{k+l} b_{d-k}')}^\circ   b_{d-k+1}' \cdots b_{r}' \Omega_f.
\end{split}
\end{equation}

Now,  for the right hand side of \eqref{Eqn=TediousClaim} we consider an expression,
\begin{equation}\label{Eqn=LastStuff1}
\begin{split}
 & ( P_{  v_{ 1 } } a_{1}  P_{  v_{ 1 } }^\perp)   \cdots   ( P_{v_{ k }} a_{ k } P_{v_{ k }}^\perp )
    (P_{v_{k+1 } }    a_{ k+1  }  P_{v_{k+1 } }  )  \cdots     (  P_{v_{k+l}}  a_{ k+l }  P_{v_{k+l}}  )   \\
   & \times (P_{v_{ k+l+1 }}^\perp  a_{  k+l+1 } P_{v_{ k+l+1 }} ) \cdots   (P_{v_{ d }}^\perp  a_{ d }  P_{v_{ d }})  b_1 \cdots b_r \Omega.
 \end{split}
 \end{equation}
 The assumption that $\sigma$ is trivial yields that  $v_{k+l+1} \cdots v_d$ starts with $V\Gamma_2$, that the letters  $v_{k}, \ldots$, $v_{k+l}$ exhaust $V\Gamma_0$ and that the letters at the end of $v_{1} \cdots v_{k}$ that commute with $\Gamma_0$ are precisely given by  $V\Gamma_1$.
 If \eqref{Eqn=LastStuff1} is non-zero then let us argue that there exists a word  $w_1'\cdots w_r'$ as defined above. Indeed, if  \eqref{Eqn=LastStuff1} is non-zero, then we may shuffle  $b_1 \cdots b_r$ into an operator $b_1' \cdots b_r'$  of type  $w_1'\cdots w_r'$ such that:  $w_1' \cdots w_{d-k-l}'$ equals  $v_{d} \cdots v_{k+l+1}$ and ends with $V\Gamma_2$;  the letters $w_{d-k-l+1}'\cdots w_{d-k}'$  exhaust  $V\Gamma_0$;  $w_{d-k+1}' \cdots  w_{r}'$ does not have a letter of $V\Gamma_1$ up front (because if that happens then applying $P_{v_i}^\perp, i \leq k$ will give zero).
   So we conclude that \eqref{Eqn=LastStuff1} can only be non-zero if there exists $w_1'\cdots w_r'$ as defined above,  in which case
 \begin{equation}\label{Eqn=LastStuff2}
 \begin{split}
\eqref{Eqn=LastStuff1} = & ( P_{v_{ 1 }}  a_{ 1 }   P_{v_{ 1 }}^\perp)   \cdots   (P_{v_{ k }}  a_{  k  }  P_{v_{ k }}^\perp)  (P_{v_{k+1 } }    a_{ k+1 }  P_{v_{k+1 } }  )  \cdots     (  P_{v_{k+l}}  a_{k+l}  P_{v_{k+l}}  ) \\
 & \qquad \times   ( P_{v_{ k+l+1 }}^\perp  a_{  k+l+1  }  P_{v_{ k+l+1 }} ) \cdots   ( P_{v_{ d }}^\perp  a_{ d  } P_{v_{ d }}) b_1' \cdots b_r' \Omega \\
 = & \langle a_{d} b_{1}' \Omega ,  \Omega \rangle \cdots \langle a_{k+l+1} b_{d-k-l}'\Omega ,  \Omega \rangle
  a_1 \cdots a_k    \overbrace{(a_{k+1} b_{d-k-l+1}')}^\circ \cdots \overbrace{(a_{k+l} b_{d-k}')}^\circ   b_{d-k+1}' \cdots b_{r}' \Omega.
\end{split}
\end{equation}
If one of the terms with a $\overbrace{\:}^\circ$ is zero, then also this term was zero in \eqref{Eqn=FirstStuff} and the proposition is proved. If none of these terms are zero, then
 the image of \eqref{Eqn=LastStuff2} under $\mathcal{R}_{l,k, \Gamma_0, \Gamma_1}$ equals \eqref{Eqn=FirstStuff} and \eqref{Eqn=LastStuff1} is in $\ker (\mathcal{R}_{l,k, \Gamma_0, \Gamma_1})^\perp$. This concludes the proposition in case $\sigma_\cQ$ exists.

 If $\sigma_\cQ$ does not exist, then
\[
\cQ_{l, d-l-k,  \Gamma_0, \Gamma_2, \Gamma_1}( b_1 \cdots b_r \Omega ) = 0.
\]
On the other hand we already noted that \eqref{Eqn=LastStuff2} can only be nonzero if a permutation  $\sigma_\cQ$ exists. So if $\sigma_{\cQ}$ is non-existent then also \eqref{Eqn=LastStuff2} is zero, yielding the proposition.

\end{proof}

 Set the product map
 \[
 \rho_d: \Sigma_d \rightarrow \cB(L^2(A, \varphi)): a_1 \otimes \cdots \otimes a_d \mapsto a_1 \cdots a_d.
 \]
We define the map for $d \in \mathbb{N}_{\geq 1}$,
\begin{equation}\label{Eqn=PiDnonfree}
\pi_d: j_d(\Sigma_d) \rightarrow \mathcal{B}(L^2(A, \varphi)): j_d(x) \mapsto \rho_d(x),
\end{equation}
so that by definition $\pi_d \circ j_d = \rho_d$. Now the crucial part is to show that the map $\pi_d$ is well-defined and completely bounded with linear bound in $d$. This is where we use the announced intertwining argument between graph products and free products.

Now we are ready for the main theorem of this section.
  Recall that the word length projection $\chi_d$ was defined in  \eqref{Eqn=WordLengthProj}.

\begin{theorem}[Graph product Khintchine inequality] \label{Thm=LongTheoremGraph}
 Let $\Gamma$ be a finite simplicial graph and consider a graph product $(A, \varphi) = \ast_{v, \Gamma} (A_v, \varphi_v)$ of unital C$^\ast$-algebras $A_v$ with GNS-faithful states $\varphi_v$.   Then for every $d \in \mathbb{N}_{\geq{1}}$ there exist  maps
\[
j_d: \chi_d( A ) \rightarrow X_d, \qquad \pi_d: \Dom(\pi_d) \subseteq  X_d \rightarrow  \chi_d( A  ),
\]
with  $\Dom(\pi_d) = j_d( \chi_d(A))$ and where $X_d$ is defined in \eqref{Eqn=Xd2} and \eqref{Eqn=IsoIso} such that:
\begin{enumerate}[label=(\roman*)]
\item\label{Item=Khin1}  $\pi_d \circ j_d$ is the identity on $\chi_d( A )$;
\item \label{Item=Khin3}  $\Vert \pi_d: \Dom(\pi_d) \rightarrow A \Vert_{cb} \leq (\# \Cliq(\Gamma))^3 d$.
\end{enumerate}
\end{theorem}

\begin{proof} Part
\ref{Item=Khin1} follows from Proposition  \ref{Prop=TProjectionDec}.
 From Proposition  \ref{Prop=GeneralKin} we see that on the domain $j_d( \chi_d(A))$ the map  $\pi_d$ is given by the direct sum of the maps
 \[
 \mathcal{R}_{l,k, \Gamma_0, \Gamma_1}^\ast \pi_{d, l,k, \Gamma_0, \Gamma_1, \Gamma_2}^f(   \: \cdot \:  )  \mathcal{Q}_{ l,d-l-k, \Gamma_0, \Gamma_2, \Gamma_1}.
    \]
    In particular, $\pi_d$ is well-defined.
    As each of these summands is completely contractive and there are at most $(\# \Cliq(\Gamma))^3 d$  summands,   we see that $\pi_d$ is completely bounded with the desired complete bound.
\end{proof}

\begin{rmk}
It is possible that a clever refinement of the present methods or perhaps an entirely different proof could yield an improvement of the constant $(\# \Cliq(\Gamma))^3 d$ in the Khintchine inequality. For the applications we have in mind, the current method cannot be altered without affecting some of the results in the remainder of the paper as they depend on the explicit form of the map $j_d$ and the space $X_d$. We leave the question of whether an improved constant can be attained to future work.
\end{rmk}

\section{Khintchine and Haagerup inequalities for right-angled Hecke C$^\ast$-algebras} \label{Inequalities}
In this section we make the Khintchine inequality from Section \ref{A graph product Khintchine-Haagerup inequality} explicit in the case of Hecke algebras. As a consequence  we derive a Haagerup inequality for right-angled Coxeter groups and their Hecke deformations. Such a Haagerup inequality shows that the $L^\infty$-norm of an operator of length $d$ can be estimated with the  $L^2$-norm up to a polynomial bound $Q(d)$. It is a   generalisation of  Haagerup's inequality for free groups $\mathbb{F}_n$, see \cite{HaagerupExample} or also \cite[Section 9.6]{Pisier}. It entails that there exists a constant $C$ such that for every $x \in \mathbb{C}[\mathbb{F}_n]$ supported on group elements of length $d \in \mathbb{N}_{\geq 1}$ we have
\[
\Vert x \Vert \leq C d \Vert x \Vert_2.
\]
In particular, here $Q(d) = d$ so that we have a linear estimate in the length $d$.
Haagerup and Khintchine inequalities  have found a wide range of applications in operator theory. We will give further applications to C$^\ast$-simplicity problems in Section \ref{simplicity}.

\vspace{0.3cm}

Let $(W,S)$ be a  finitely generated  right-angled Coxeter system and let $q = (q_s)_{s \in S} \in \mathbb{R}_{> 0}^{(W,S)}$ be a multi-parameter set.  Let $\Gamma$ be the  graph associated to $(W,S)$ defined in Section \ref{Sect=GraphCoxeter}.
From  \eqref{Eqn=GraphHeckeIso}  we see that we have a canonical isomorphism
\[
(C_{r,q}^\ast(W), \tau_q) \simeq \ast_{s, \Gamma} (C_{r,q_s}^\ast(W_s), \tau_{q_s}).
\]
  Furthermore, we may specialize Section \ref{A graph product Khintchine-Haagerup inequality} to $\ast_{s, \Gamma} (C_{r,q_s}^\ast(W_s), \tau_{q_s})$. In particular this defines the space of diagonal operators $A_{\Gamma_0}$ as in \eqref{Eqn=AGammaNot} and the operator space $X_d$ of \eqref{Eqn=Xd2}. We first observe that  $A_{\Gamma_0}$ simplifies. Recall that $P_s, s \in S$ and $P_{\Gamma_0}^f, \Gamma_0 \in \Cliq(\Gamma)$ were defined in \eqref{Eqn=ProjectionSpace} and \eqref{Eqn=PGammaF}.

\begin{lem}\label{Lem=Simple}
For a right-angled Coxeter group $(W,S)$ we have $A_{\Gamma_0} = \mathbb{C} P_{\Gamma_0}^f$ for any $\Gamma_0 \in \Cliq(\Gamma,l)$. Moreover, for $\boldv = s_1 \cdots s_l, s_i \in S$ a reduced word in $I$ with $\{s_1, \ldots, s_l \} = V\Gamma_0$ we have,
 \[
 \Diag \left( P_{s_{1}} T_{s_{1}}^{(q)} P_{s_{1}},  \ldots,  P_{s_{l}}  T_{s_{l}}^{(q)} P_{s_{l}} \right)   = (\prod_{s \in V\Gamma_0}  p_s(q)) P_{\Gamma_0}^f.
 \]
\end{lem}
\begin{proof}
For $s \in S$ we have that $P_s T_s^{(q)} P_s = p_s(q) P_s$  and $C_{r,q}^\ast(W_s)$ is a two dimensional C$^\ast$-algebra spanned by the identity and $T_s^{(q)}$ (see the comments after \cite[Corollary 3.3]{CaspersAPDE}). Now let $\boldv = s_1 \cdots s_l$ be as in the lemma. Let $A_{s} = C_{r, q_s}^\ast(W_s)$. For operators $a_i \in A_{s_i}, 0 \leq i \leq l$, we then have that $P_{s_i} a_i P_{s_i}$ is a scalar multiple of $P_{s_i}$ so that $\Diag(   a_{1}   ,\ldots ,   a_{l} )$  is a scalar multiple of $P_{\Gamma_0}^f$, see \eqref{Eqn=DiagExplain}. If $a_i = T_s^{(q)}$ the scalar multiple is $\prod_{s \in V\Gamma_0}  p_s(q)$.
\end{proof}

Lemma \ref{Lem=Simple} shows that we may identify $A_{\Gamma_0} \simeq \mathbb{C}$ completely isometrically.  Then in \eqref{Eqn=IsoIso} note that  $M_{k, d-k-l}(\mathbb{C}) \hotimes \mathbb{C} = M_{k, d-k-l}(\mathbb{C})$.
So for a right-angled Coxeter system we get by Lemma \ref{Lem=Simple}, \eqref{Eqn=Xd2} and \eqref{Eqn=IsoIso},
 \begin{equation}\label{Eqn=Xd3}
X_d =  \bigoplus_{l=0}^{d}   \bigoplus_{k=0}^{d-l}   \bigoplus_{\Gamma_0 \in {\rm Cliq}(\Gamma, l)} \bigoplus_{(\Gamma_1, \Gamma_2) \in \Comm(\Gamma_0)}  M_{k, d-k-l}(\mathbb{C}).
\end{equation}
Let $p_{l,k, \Gamma_0, \Gamma_1, \Gamma_2}$ be the projection of $X_d$ onto the summand   $M_{k, d-k-l}(\mathbb{C})$  indexed by $(l,k, \Gamma_0, \Gamma_1, \Gamma_2)$.
 We equip $M_{k, d-k-l}(\mathbb{C})$ with the inner product
\[
\langle x, y  \rangle_{\Tr} = (  {\rm Tr}_{d-k-l}   )(y^\ast x),
\]
where ${\rm Tr}_{d-k-l}$ is the non-normalized trace that takes the value 1 on rank 1 projections.
We further equip $X_d$ with the direct sum of these inner products. For $x \in X_d$ we let
\[
\Vert x \Vert_{2, \Tr} = \langle x, x \rangle^{\frac{1}{2}}_{\Tr}.
\]
Then, as for any finite dimensional type I von Neumann algebra, we have
\begin{equation}\label{Eqn=LinftyL2Ineq}
\Vert x \Vert \leq \Vert x \Vert_{2, \Tr}.
\end{equation}
By Theorem  \ref{Thm=LongTheoremGraph} we obtain maps
\[
j_d: \chi_d(   C_{r,q}^\ast(W)    ) \rightarrow X_d,  \qquad \textrm{ and }  \qquad  \pi_d: \Dom(\pi_d) \subseteq X_d \rightarrow \chi_d(C_{r,q}^\ast(W)),
\]
with $\Dom(\pi_d)  = j_d (\chi_d(   C_{r,q}^\ast(W)    ))$   such that $\pi_d \circ j_d$ is the identity  on $C_{r,q}^\ast(W)$ and
\[
\Vert \pi_d: \Dom(\pi_d) \rightarrow X_d \Vert_{cb} \leq  \# \Cliq(\Gamma)^3 d.
\]
We now have the following orthogonality lemma.

\begin{lem}
Let $d \in \mathbb{N}_{\geq 1}, 0 \leq l \leq d, 0 \leq k \leq d-l, \Gamma_0 \in {\rm Cliq}(\Gamma, l), (\Gamma_1, \Gamma_2) \in \Comm(\Gamma_0)$. Let $\boldv, \boldw \in W$ be reduced of length $d$. If the permutation $\sigma^{\boldv}$ given in Definition \ref{Dfn=Sigma} exists we have
  \begin{equation}\label{Eqn=ClaimOrth}
    \langle p_{l,k, \Gamma_0, \Gamma_1, \Gamma_2} j_d(T_\boldv^{(q)} ),  j_d(T_\boldw^{(q)} ) \rangle_{ {\rm Tr} } =
      \delta(\boldv \simeq  \boldw) \prod_{s \in V \Gamma_0 } p_{s}(q)^2.
  \end{equation}
  If such $\sigma^{\boldv}$ does not exist then   $p_{l,k, \Gamma_0, \Gamma_1, \Gamma_2} j_d(T_\boldv^{(q)} ) = 0$.
  \end{lem}
\begin{proof}
The final claim of the statement follows from the definition of $j_d$. So it remains to prove \eqref{Eqn=ClaimOrth} and we assume that $\sigma^{\boldv}$ as in Definition \ref{Dfn=Sigma} exists. Since the right hand side of \eqref{Eqn=ClaimOrth} has the term $\delta(\boldv  \simeq \boldw)$ we may assume that also $\sigma^{\boldw}$ exists. Moreover, by shuffling the letters of $\boldv$ and $\boldw$ if necessary, which does not change the operators $T_\boldv^{(q)}$ and $T_\boldw^{(q)}$, we may assume that $\sigma^\boldv$ and $\sigma^\boldw$ are the  identity permutation.

For $s \in S$ set
\[
e_s =  P_s^f  T_s^{(q)} P_s^{f,\perp}  \quad \textrm{ and } \quad e_s' =  P_s^{f,\perp}  T_s^{(q)} P_s^f.
\]
 These form an orthonormal basis of the respective column Hilbert space $L_1$ and row Hilbert space $K_1$. Now write a reduced expression $\boldv = s_1 \cdots s_d, s_i \in S$. By assumption that $\sigma^\boldv$ was trivial we see that  $s_{k+1}, \ldots, s_{k+l}$ commute and   form a clique $\Gamma_0$ in $\Gamma$.
 From Lemma \ref{Lem=Simple} we find that
 \[
 \Diag \left( P_{s_{k+1}} T_{s_{k+1}}^{(q)} P_{s_{k+1}},  \ldots,  P_{s_{k+l}}  T_{s_{k+l}}^{(q)} P_{s_{k+l}} \right)   = (\prod_{s \in V\Gamma_0}  p_s(q)) P_{\Gamma_0}^f.
 \]
  It now follows from the definition of $j_d$ that
 \[
 p_{l,k, \Gamma_0, \Gamma_1, \Gamma_2}  j_d(T_{\boldv}^{(q)}) = \left( \prod_{s \in V\Gamma_0} p_s(q) \right)  e_{s_1} \otimes \cdots \otimes e_{s_k}   \otimes e_{s_{k+l+1}}' \otimes \cdots \otimes e_{s_d}'.
 \]
From this we can directly conclude \eqref{Eqn=ClaimOrth}.
\end{proof}


\begin{theorem}[Khintchine inequality for right-angled Hecke C$^\ast$-algebras] \label{Thm=LongTheorem}
Let $(W,S)$ be a    right-angled Coxeter system with finite generating set $S$ and graph $\Gamma$. Let $q = (q_s)_{s \in S} \in \mathbb{R}_{> 0}^{ (W,S)  }$ be a multi-parameter.   Then for every $d \in \mathbb{N}_{\geq 1}$ there exist maps
\[
j_d: \chi_d( C_{r,q}^\ast(W)   ) \rightarrow    X_d, \qquad \pi_d: \Dom(\pi_d) \subseteq    X_d \rightarrow  \chi_d(   C_{r,q}^\ast(W)   ),
\]
with  $j_d(\chi_d( C_q^\ast(W)) = \Dom(\pi_d)$ and $X_d$ defined in \eqref{Eqn=Xd3} such that:
\begin{enumerate}[label=(\roman*)]
\item\label{Item=Khin1}  $\pi_d \circ j_d$ is the identity on $\chi_d( C_{r,q}^\ast(W))$;
\item \label{Item=Khin3}  $\Vert \pi_d: \Dom(\pi_d)\rightarrow C_{r,q}^\ast(W)\Vert_{cb} \leq (\# \Cliq(\Gamma))^3 d$;
\item \label{Item=Khin2}   $j_d$ extends to a bounded map
\[
L^2(\chi_d( C_{r,q}^\ast(W)), \tau_q ) \rightarrow L^2(X_d, \Tr),
\]
 with bound majorized by $\prod_{s \in S} p_s(q)$.
\end{enumerate}
\end{theorem}
\begin{proof}
Statements \ref{Item=Khin1} and \ref{Item=Khin3} are immediate from \eqref{Eqn=GraphHeckeIso} and Theorem \ref{Thm=LongTheoremGraph}. We thus prove  \ref{Item=Khin2}. Let $\boldv \in I$ have length $d$.    Since $\Vert T_\boldv^{(q)} \Omega \Vert_2 = 1$ we find from \eqref{Eqn=ClaimOrth} that
  \[
 \Vert j_d: L^2(\chi_d( C_{r,q}^\ast(W) ), \tau_q) \rightarrow L^2(X_d, \Tr_{X_d})\Vert \leq \sup_{\boldv \in W, \vert \boldv \vert = d}  \Vert j_d( T_\boldv^{(q)})  \Vert_{2, \Tr} \leq \prod_{s \in S} p_s(q).
  \]
  This completes the proof.
\end{proof}

\begin{theorem}[Haagerup inequality for right-angled Hecke C$^\ast$-algebras]\label{Thm=Khintchine}
Let $(W, S)$ be a right-angled Coxeter group with finite generating set $S$.  Let $q = (q_s)_{s \in S} \in \mathbb{R}_{> 0}^{ (W,S)}$ be a multi-parameter. Then for each $d \in \mathbb{N}_{\geq 1}$ and  $x \in \chi_{d}(C_{r,q}^\ast(W))$ we have
\[
\Vert x \Vert \leq d  (\# \Cliq(\Gamma))^3 ( \prod_{s \in S} p_s(q) )  \Vert x \Vert_2.
\]
\end{theorem}
\begin{proof}
By respectively Theorem \ref{Thm=LongTheorem} \ref{Item=Khin1}, \eqref{Eqn=LinftyL2Ineq} and Theorem \ref{Thm=LongTheorem} \ref{Item=Khin3}, \ref{Item=Khin2},  we get for every  $x \in \chi_d( C_{r,q}^\ast(W) )$,
\[
\begin{split}
 \Vert x \Vert = & \Vert (\pi_d \circ j_d)(x)\Vert \leq \Vert \pi_d \Vert \Vert j_d(x) \Vert
 \leq  \Vert \pi_d \Vert \Vert j_d(x) \Vert_{2, \Tr} \\  \leq&    d  (\# \Cliq(\Gamma))^3  \Vert j_d(x) \Vert_{2, \Tr} \leq   d  (\# \Cliq(\Gamma))^3 ( \prod_{s \in S} p_s(q) )  \Vert x \Vert_{2}.
\end{split}
\]
This completes the proof.
\end{proof}


\section{On isomorphisms of Hecke algebras and their C$^\ast$- and von Neumann algebras}\label{Sect=Iso}
In this section we discuss  isomorphism properties of Hecke algebras and in particular their dependence on $q$. These properties are well understood for finite   Coxeter systems, which we summarize as follows.

\begin{rmk}
By Tits's deformation theorem, the Hecke algebras $\mathbb{C}_{q}\left[W\right]$ of a given finite Coxeter system $\left(W,S\right)$ are pairwise isomorphic to each other for different $q \in \mathbb{R}_{>0}^{(W,S)}$,  see \cite[Proposition 10.11.2]{Carter}. The argument is based on the classification of finite-dimensional semi-simple algebras and the isomorphism is not explicit.
\end{rmk}

To the knowledge of the authors there is no general statement for arbitrary Coxeter systems $(W,S)$ about the dependence  on $q \in \mathbb{R}_{>0}^{(W,S)}$ of the isomorphism class of their Hecke deformations.  However, in the case of right-angled (not necessarily finite) Coxeter systems one can still prove that all Hecke deformations are isomorphic; even with an explicit isomorphism. See e.g. \cite[(2.1.13)]{Matsumoto}, \cite[Corollary 9.7]{ScottOkun} for this result which we present in an alternative way and which is suited   for the next sections.

\begin{proposition} \label{isomorphism}
Let $\left(W,S\right)$ be a right-angled Coxeter system and $q=\left(q_{s}\right)_{s\in S}\in\mathbb{R}_{>0}^{(W,S)}$. Then the map $\pi_{q,1}\text{: }\mathbb{C}_{1}\left[W\right]\rightarrow\mathbb{C}_{q}\left[W\right]$ given by
\begin{equation} \label{map}
1\mapsto 1 \quad \text{ and } \quad  T_{s}^{\left(1\right)}\mapsto\frac{1-q_{s}}{1+q_{s}}+\frac{2\sqrt{q_{s}}}{1+q_{s}}T_{s}^{\left(q\right)}
\end{equation}
for $s\in S$ defines an isomorphism of $\ast$-algebras.
\end{proposition}

\begin{proof}
Set $\alpha_{s}\left(q\right):=\left(1-q_{s}\right)/\left(1+q_{s}\right)$ and $\beta_{s}\left(q\right):=2\sqrt{q_{s}}/\left(1+q_{s}\right)$. Being its own inverse, for $s\in S$ the expression $\alpha_{s}\left(q\right)+\beta_{s}\left(q\right)T_{s}^{(q)}\in\mathbb{C}_{q}\left[W\right]$ is invertible. Hence we get a map $S\rightarrow\mathbb{C}_{q}\left[W\right]^{\times}$, $s\mapsto\alpha_{s}\left(q\right)+\beta_{s}\left(q\right)T_{s}^{(q)}$ that uniquely extends to a group homomorphism $\phi$ on the free group $F\left(S\right)$ in $S$. Since $\left(W,S\right)$ is right-angled one easily checks that $\phi \left(st\right)^{m_{s,t}}=\left(\phi\left(s\right)\phi\left(t\right)\right)^{m_{s,t}}$ for all $s,t\in S$. This implies that $\phi$ induces a group homomorphism $\phi': W\rightarrow\mathbb{C}_{q}\left[W\right]^{\times}$ with $(\phi'(\boldw))^\ast=\phi\left(\boldw^{-1}\right)$ for every $\boldw \in W$. The universal property of the group algebra $\mathbb{C}_{1}\left[W\right]$ then implies the existence of the unital $\ast$-algebra homomorphism $\pi_{q,1}$. It is clearly surjective. The injectivity follows from the universal property of the Hecke algebra $\mathbb{C}_{1}\left[W\right]$.
\end{proof}

\begin{remark}
The homomorphism prescribed by \eqref{map} does not necessarily exist if $\left(W,S\right)$ is not right-angled. This already fails for the Coxeter system $\left(W,S\right)$ with $S=\left\{ s,t\right\}$  and $m_{s,s}=m_{t,t}=2$, $m_{s,t}=3$ and points out an inaccuracy in \cite[Section 19, Note 19.2 on p. 358]{Da}.
\end{remark}

\begin{rmk}
For a right-angled Coxeter system $(W,S)$ with $\vert S \vert \geq 2$ the isomorphism of Proposition \ref{isomorphism} does not extend to an isomorphism $C_{r,1}^\ast(W) \rightarrow C_{r,q}^\ast(W)$ for all $q \in \mathbb{R}_{>0}^{(W,S)}$. Indeed, in Example \ref{example1} we show that this cannot be the case. In the case $\vert S \vert \geq 3$ this can also be proved through the simplicity of  $C_{r,q}^\ast(W)$ in the same way as Remark \ref{Rmk=Neuman}, see Section \ref{Sect=Simplicity}.
\end{rmk}

\begin{rmk}\label{Rmk=Neuman}
Garncarek's factoriality result of Theorem \ref{Thm=Garncarek} illustrates that the isomorphism of Proposition \ref{isomorphism} does not necessarily extend to an isomorphism of the corresponding Hecke-von Neumann algebras. Indeed, for $q \notin \left[ \rho,\rho^{-1} \right]$ the Hecke-von Neumann algebra $\mathcal{N}_q(W)$ is not a factor whereas $\mathcal{N}_{1}(W)$ is a factor. Hence there can be no isomorphism at all between $\mathcal{N}_q(W)$ and $\mathcal{N}_1(W)$. This was already observed in \cite[Section 19, Note 19.2 on p. 358]{Da}. The situation is even more delicate, see the next Remark \ref{Rmk=FreeFactorProblem}.
\end{rmk}

\begin{rmk}[Free factor problem]\label{Rmk=FreeFactorProblem}
 Consider the right-angled Coxeter group $W = (\mathbb{Z}_2)^{\ast l}, l \geq 3$. By Theorem \ref{Thm=Garncarek} we see that  for the single parameter $q \in [\frac{1}{l-1}, 1]$ we have that $\cN_q(W)$ is  a II$_1$-factor. Moreover, \cite{DykemaInter}  together with a calculation in   \cite[Section 6]{Gar} shows that for $q \in [\frac{1}{l-1}, 1]$ we have that $\cN_q(W) \simeq \cL( \mathbb{F}_{  2lq (1 + q)^{-2} }      )$ where $\mathcal{L}(\mathbb{F}_t), t \in \mathbb{R}_{>1}$ is the interpolated free group factor, cf. \cite{DykemaInter}, \cite{Radulescu}. By \cite{DykemaInter}, \cite{Radulescu}, the interpolated free group factors are either all isomorphic or they are all non-isomorphic. The problem which of the two in this dichotomy is true is known as the famous free factor problem. Hence, solving the isomorphism question of $\cN_q(W)$ for different $q \in  [\frac{1}{l-1}, 1]$ is equivalent to the free factor problem.  In Section \ref{Sect=Simplicity} we shall show that for $q  \in [\frac{1}{l-1}, 1]$ we have that $C_{r,q}^\ast(W)$ has unique trace; therefore if any  two  $C_{r,q}^\ast(W)$ with $q  \in [\frac{1}{l-1}, 1]$ are isomorphic we get by Lemma \ref{Lem=StandardExt} that two of the von Neumann algebras $\cN_q(W), q \in [\frac{1}{l-1}, 1]$ would be isomorphic. Since solving the free factor problem using these C$^\ast$-algebraic methods  seems unrealistic\footnote{And solving it in the affirmative using C$^\ast$-algebras seems even more unrealistic.} we believe that all  $C_{r,q}^\ast(W)$ with $q  \in [\frac{1}{l-1}, 1]$ are non-isomorphic.
\end{rmk}

We finish this section with the following lemma which is well-known on the algebraic level, see for instance \cite[Section 9]{ScottOkun}. For convenience of the reader we include the proof for the associated C$^\ast$- and von Neumann algebras here.

\begin{proposition} \label{unitary}
Let $\left(W,S\right)$ be a Coxeter system and $q=\left(q_{s}\right)_{s\in S}\in  \RWS$, $\epsilon=\left(\epsilon_{s}\right)_{s\in S}\in \OneWS$. Then $C_{r,q}^{\ast}\left(W\right) \simeq C_{r,q'}^{\ast}\left(W\right)$ via $T_{s}^{\left(q\right)}\mapsto\epsilon_{s}T_{s}^{\left(q'\right)}$ where $q':=\left(q_{s}^{\epsilon_{s}}\right)_{s\in S}\in\mathbb{R}_{>0}^{(W,S)}$.
\end{proposition}
\begin{proof}
Note that $\epsilon_s p_s(q') = p_s(q)$. Then from the defining properties of a Hecke algebra \eqref{Eqn=Rel1} and \eqref{Eqn=Rel2} we have that $T_s^{(q)} \mapsto \epsilon_s T_s^{(q')}, s \in S$ determines a $\ast$-isomorphism $\pi_{q', q}: \mathbb{C}_q[W] \rightarrow \mathbb{C}_{q'}[W]$. Moreover,
\[
\tau_{q'} \circ \pi_{q', q}(T_{\boldw}^{(q)}) = \epsilon_{\boldw} \tau_{q'}(  T_{\boldw}^{(q')}) = \epsilon_{\boldw} \delta(\boldw = \emptyset) = \delta(\boldw = \emptyset) =   \tau_q(T_{\boldw}^{(q)}),
\]
 so that $ \pi_{q', q}$ is trace preserving. By Lemma \ref{Lem=StandardExt} $\pi_{q', q}$ extends to a $\ast$-isomorphism $C_{r,q}^\ast(W) \rightarrow C_{r,q'}^\ast(W)$.
\end{proof}

\section{Simplicity of Hecke C$^\ast$-algebras}\label{Sect=Simplicity}

A C$^{\ast}$-algebra is called \textbf{simple} if it does not contain any non-trivial closed two-sided ideal. In this section we investigate the (non-)simplicity of Hecke C$^{\ast}$-algebras.

\begin{remark}
Note that in order to study the simplicity of Hecke C$^\ast$-algebras it suffices to consider irreducible Coxeter systems. Indeed, if $\left(W,S\right)$ is a Coxeter system that is not irreducible, it admits a non-trivial decomposition of the form $\left(W,S\right)=\left(W_{T}\times W_{T'},T\cup T'\right)$. For every $q \in \mathbb{R}_{>0}^{(W,S)}$ this induces a decomposition of the corresponding Hecke algebra into an algebraic tensor product \begin{equation} \nonumber \mathbb{C}_{q}\left[W\right]\simeq\mathbb{C}_{q}\left[W_{T}\right]\odot \mathbb{C}_{q}\left[W_{T'}\right] \text{.} \end{equation} By Lemma \ref{Lem=StandardExt} this extends to the C$^{\ast}$- and von Neumann algebraic level \begin{equation} \nonumber C_{r,q}^{\ast}\left(W\right)\simeq C_{r,q}^{\ast}\left(W_{T}\right)\otimes C_{r,q}^{\ast}\left(W_{T'}\right)\text{, }{\mathcal N}_{q}\left(W\right)\simeq {\mathcal N}_{q}\left(W_{T}\right) \overline{\otimes}{\mathcal N}_{q}\left(W_{T'}\right)\text{.}\end{equation} Since the (spatial) tensor product of two C$^{\ast}$-algebras is simple if and only if both C$^{\ast}$-algebras are simple, it suffices to look at irreducible Coxeter systems.
\end{remark}

Recall Garncarek's characterisation of factoriality of the single-parameter Hecke-von Neumann algebras of right-angled Coxeter groups, see Theorem \ref{Thm=Garncarek}.
Another result related to the simplicity of Hecke C$^{\ast}$-algebras appears in \cite{DeLaHarpe} (see also \cite{Fe} and \cite{Cornulier}).

\begin{theorem}
An irreducible Coxeter group which is neither of spherical nor affine type is C$^{\ast}$-simple and has unique tracial state.
\end{theorem}
In the remaining sections we give partial answers to the question for a characterisation of those Hecke C$^{\ast}$-algebras that are simple and have unique tracial state. In the case of a free product of abelian Hecke C$^\ast$-algebras we obtain a complete answer.


\subsection{Non-simplicity of Hecke C$^\ast$-algebras}

Let $\left(W,S\right)$ be a Coxeter system with $\left|S\right|<\infty$ and $z:=\left(z_s\right)_{s\in S}\in \mathbb{C}^{(W,S)}$. For every reduced expression $\mathbf{w}=s_{1}...s_{n}$ of $\mathbf{w}\in W$ define $z_{\mathbf w}:=z_{s_{1}}...z_{s_{n}}$. The growth series of $W$ is the power series in $z$ defined by
\[
W\left(z\right):=\sum_{\mathbf{w}\in W}z_{\mathbf w}.
\]
 We denote its region of convergence by ${\mathcal R}:=\left\{ z\in \mathbb{C}^{(W,S)} \mid W\left(z\right)\text{ converges}\right\}$. For more information on the growth series see \cite[Chapter 17]{Da}. We further set
\begin{eqnarray} \label{generalized radius}
\mathcal{R}':=\left\{ \left(q_{s}^{\epsilon_{s}}\right)_{s\in S}\mid q\in{\mathcal R}\cap \mathbb{R}_{>0}^{(W,S)}\text{, }\epsilon\in \left\{-1,1\right\}^{(W,S)}\right\} \text{.}
\end{eqnarray}
Denote  the closure of ${\mathcal R'}$ in $\mathbb{R}_{>0}^{(W,S)}$ by $\overline{{\mathcal R}'}$.

\begin{lemma} \label{non-simplicity1}
Let $\left(W,S\right)$ be a Coxeter system with $\left|S\right|<\infty$ and $q\in\overline{{\mathcal R'}}$. Then there exists a character on $C_{r,q}^{\ast}\left(W\right)$. In particular $C_{r,q}^{\ast}\left(W\right)$ is not simple and does not have unique tracial state.
\end{lemma}

\begin{proof}
 First assume that $q\in{\mathcal R}'$ with $q_s \leq 1$ for all $s \in S$. As shown in the proof of \cite[Theorem 5.3]{Gar}, there exists a central projection $E_{q}\in{\mathcal N}_{q}\left(W\right)$ with $E_{q}T_{\mathbf w}^{\left(q\right)}=q_{\mathbf w}^{1/2}E_{q}$ for every $\mathbf{w}\in W$ (compare also with \cite[Lemma 19.2.5]{Da} but note that our notational conventions differ slightly). Hence the map $\chi_{q}(\: \cdot \:):=\tau_q\left(\: \cdot \: E_{q}\right)/\left\Vert E_{q}\right\Vert _{2}$ is a character on $C_{r,q}^{\ast}\left(W\right)$ with $\chi_{q}(T_{\mathbf{w}}^{\left(q\right)})=q_{\mathbf{w}}^{1/2}$, $\mathbf{w} \in W$. Its kernel is a non-trivial maximal ideal in $C_{r,q}^{\ast}\left(W\right)$.

For $q\in\overline{\mathcal{R}'}\setminus\mathcal{R}'$ with $q_s \leq 1$ for all $s \in S$ choose a sequence $\left(q_{n}\right)_{n\in\mathbb{N}}\subseteq\mathcal{R}$ with $q_{n}\rightarrow q$. The map $\chi\text{: }T_{\mathbf{w}}^{\left(q\right)}\mapsto q_{\mathbf{w}}^{1/2}$ defines a character on $\mathbb{C}_{q}\left[W\right]$. For the finite sums $x:=\sum_{\mathbf{w}\in W}x\left(\mathbf{w}\right)T_{\mathbf{w}}^{\left(q\right)}\in\mathbb{C}_{q}\left[W\right]$ and $x_{n}:=\sum_{\mathbf{w}\in W}x\left(\mathbf{w}\right)T_{\mathbf{w}}^{\left(q_{n}\right)}\in\mathbb{C}_{q_{n}}\left[W\right]$ we have $x_{n}\rightarrow x$ in $\mathcal{B}\left(\ell^{2}\left(W\right)\right)$ and $\chi_{q_{n}}\left(x_{n}\right)\rightarrow\chi\left(x\right)$ with $\chi_{q_{n}}$ defined as above. This implies
\begin{eqnarray}
\nonumber
\left|\chi\left(x\right)\right|=\lim_{n\rightarrow\infty}\left|\chi_{q_{n}}\left(x_{n}\right)\right|\leq\lim_{n}\left\Vert x_{n}\right\Vert =\left\Vert x\right\Vert \text{,}
\end{eqnarray}
so $\chi$ extends to a character on $C_{r,q}^{\ast}\left(W\right)$. Hence $C_{r,q}^{\ast}\left(W\right)$ is not simple and does not have unique tracial state.

For general $q\in \overline{\mathcal{R'}}$ the statement follows from the above in combination with Proposition \ref{unitary}
\end{proof}

It follows from Lemma \ref{non-simplicity1} that Hecke C$^{\ast}$-algebras coming from irreducible Coxeter systems of spherical or affine type are never simple and  never have unique tracial state, see Corollary \ref{non-simplicity2}. We are aware of the fact that affine type Coxeter systems always give rise to type I Hecke C$^{\ast}$-algebras (see \cite[\S 4]{Matsumoto}), which implies  non-simplicity. However, the proof of this statement makes use of the (non-trivial) characterisation of affine Coxeter groups in terms of affine Weyl groups and a resulting description of the corresponding Hecke algebras. The proof we present here is  elementary.

\begin{corollary} \label{non-simplicity2}
Let $\left(W,S\right)$ be an irreducible Coxeter system of spherical or affine type. Then for any choice of parameter $q \in \mathbb{R}_{>0}^{(W,S)}$there exists a character on the corresponding Hecke C$^{\ast}$-algebra $C_{r,q}^\ast(W)$. In particular, $C_{r,q}^\ast(W)$ is not simple and does not have unique tracial state.
\end{corollary}

\begin{proof}
First assume that $\left|S\right|<\infty$ and let $q \in \mathbb{R}_{>0}^{(W,S)}$ with $q_s \leq 1$ for every $s\in S$. Being of spherical or affine type the Coxeter group $W$ is amenable. Hence by \cite[Proposition 17.2.1]{Da} the radius of convergence of the power series $W(z)$ is one. But then $q\in\overline{{\mathcal R}}\cap\mathbb{R}_{>0}^{(W,S)}$, so $C_{r,q}^{\ast}\left(W\right)$ admits a character by Lemma \ref{non-simplicity1}.
For general $q \in \mathbb{R}_{>0}^{(W,S)}$, the statement follows with Proposition \ref{unitary}.

Next assume that $\left|S\right|=\infty$ and so $\left(W,S\right)$ is of spherical type. Again, the map $\chi\text{: }T_{\mathbf{w}}^{\left(q\right)}\mapsto q_{\mathbf{w}}^{1/2}$, $\mathbf{w}\in W$ defines a character on $\mathbb{C}_{q}\left[W\right]$. For every element $x:=\sum_{\mathbf{w}\in W}x\left(\mathbf{w}\right)T_{\mathbf{w}}^{\left(q\right)}\in\mathbb{C}_{q}\left[W\right]$ there exists a finite subset $T\subseteq S$ such that the support $\left\{ \mathbf{w}\in W\mid x\left(\mathbf{w}\right)\neq0\right\}$  of $x$ is contained in the finite subgroup $W_{T}$ of $W$ generated by $T$. Now $W_T$ is also a Coxeter group with the same exponents as W, see \cite[Theorem 4.1.6 (i), Theorem 3.4.2 (i)]{Da}. As in \cite[Lemma 19.2.2]{Da} one sees that $C_{r,q}^{\ast}\left(W_{T}\right)$ embeds into $C_{r,q}^{\ast}\left(W\right)$ canonically. Under this identification we have $x\in C_{r,q}^{\ast}\left(W_{T}\right)$. But the map $T_{\mathbf{w}}^{\left(q\right)}\mapsto q_{\mathbf{w}}^{1/2}$ is a character on the finite-dimensional C$^{\ast}$-algebra $C_{r,q}^{\ast}\left(W_{T}\right)$, hence $\left\Vert \chi\left(x\right)\right\Vert \leq\left\Vert x\right\Vert$. As this holds for every $x\in\mathbb{C}_{q}\left[W\right]$, $\chi$ extends to a character on $C_{r,q}^{\ast}\left(W\right)$.
\end{proof}


\subsection{Averaging operators} \label{Averaging operators}

To show simplicity properties of Hecke C$^{\ast}$-algebras coming from right-angled, irreducible Coxeter systems with at least three generators (see subsection \ref{simplicity}), we will make use of a method inspired by Power's averaging argument in \cite{Powers}. This requires the introduction of suitable averaging operators on the corresponding Hecke algebra which average over a finite subset of the Coxeter group. The statements of this subsection apply in greater generality. Therefore we will formulate our approach for arbitrary discrete groups.

Let $G$ be a discrete group. Denote the left regular representation of $G$ by $\lambda$, let $\tau$ be the canonical tracial state on $C_{r}^\ast(G)$ and fix some finite set $F\subseteq G$. Every such $F$ defines an averaging operator $\Phi$ on $C_{r}^{\ast}\left(G\right)$ given by
\begin{equation}
\nonumber
\Phi\left(x\right):=\frac{1}{\left|F\right|}\sum_{g\in F}\lambda_{g^{-1}}x\lambda_{g}\text{.}
\end{equation}
The map $\Phi$ is trace-preserving, unital and completely positive. In particular it induces a bounded operator $\tilde{\Phi}$ on $\ell^2(G)\ominus\mathbb{C}\delta_{e}$ given by $\tilde{\Phi}\left(x\delta_{e}\right):=\Phi\left(x\right)\delta_{e}$ for all $x\in\mathbb{C}\left[G\right]$ with $\tau\left(x\right)=0$.

\begin{proposition} \label{operator norm}
Let $G$ be a discrete group with the property that there exist three elements $g_1,g_2,g_3 \in G$ and a subset $D \subseteq G\setminus\left\{e\right\}$ such that $D\cup g_1Dg_1^{-1} =G\setminus\left\{e\right\}$ and the sets $D$, $g_2Dg_2^{-1}$ and $g_3 D g_3^{-1}$ are pairwise disjoint. Take $F:=\left\{e,g_1,g_2,g_3\right\}$. Then the operator $\tilde{\Phi}$ on $\ell^2(G)\ominus\mathbb{C}\delta_{e}$ has norm strictly smaller than one.
\end{proposition}

The proof of Proposition \ref{operator norm} is based on Ching's following variation of Puk\'{a}nszky's $14\varepsilon$-argument in \cite[Lemma 10]{Pukanszky}.

\begin{lemma} \cite[Lemma 4]{Ching} \label{14-lemma}
Let $G$ be a group as in Proposition \ref{operator norm}. Then
\begin{equation} \label{Pukanszky}
\nonumber
\left\Vert x-\tau\left(x\right)\right\Vert _{2}\leq14\max_{i=1,2,3}\Vert x-\lambda_{g_i^{-1}}x\lambda_{g_i}\Vert _{2}
\end{equation} for every $x\in C_{r}^{\ast}\left(G\right)$.
\end{lemma}

Further, we shall need that for arbitrary vectors $\xi_0, \ldots, \xi_n$ in a Hilbert space $\mathcal H$ we have the following equality, which for $n = 1$ is known as the parallelogram law:
\begin{equation}\label{Eqn=Parallel}
\Vert \sum_{i =0}^n \xi_i \Vert^2 + \sum_{0 \leq i < j \leq n} \Vert \xi_i - \xi_j \Vert^2 = (n+1) \sum_{i=0}^n \Vert \xi_i \Vert^2.
\end{equation}
Note that \eqref{Eqn=Parallel} can be verified directly by writing out all norms as inner products.

\begin{proof}[Proof of Proposition \ref{operator norm}]
The norm of $\tilde{\Phi}: \ell^2(G) \ominus \mathbb{C} \delta_e \rightarrow \ell^2(G) \ominus \mathbb{C} \delta_e$ is clearly majorized by 1. Now suppose that $\Vert \tilde{\Phi} \Vert = 1$. Take a sequence $x_k \in \mathbb{C}[G]$ with $\tau(x_k) = 0$ and $\Vert x_k \delta_e \Vert = 1$ such that $\Vert \tilde{\Phi}(x_k \delta_e)\Vert \nearrow 1$. Set $\xi_{i}^k = \lambda_{g_i^{-1}}  x_k  \lambda_{g_i} \delta_e$ with $g_1, g_2, g_3$ as in the proposition and $g_0$ the identity. By \eqref{Eqn=Parallel} we have
\[
\sum_{0 \leq i < j \leq 3} \Vert \xi_i^k  - \xi_j^k \Vert^2  = 4 \sum_{i=0}^3 \Vert \xi_i^k \Vert^2 - \Vert \sum_{i =0}^3 \xi_i^k \Vert^2 = 4^2 - \Vert 4 \tilde{\Phi}(x_k \delta_e) \Vert^2
\rightarrow  4^2 - 4^2 =0.
\]
Therefore each of the individual summands on the left hand side converge to 0 as $k \rightarrow \infty$. Since $\xi_0^k = x_k \delta_e$  and $\Vert y \Vert_2 = \Vert y \delta_e \Vert, y \in \mathbb{C}[G]$ we see from Lemma \ref{14-lemma} that $\Vert x_k \Vert_2 = \Vert x_k - \tau(x_k) \Vert_2 \rightarrow 0$. This contradicts that $\Vert x_k \delta_e \Vert =1$. We conclude that $\Vert \tilde{\Phi} \Vert < 1$.
  \end{proof}


\subsection{Simplicity in the right-angled case for $q$ close to 1} \label{simplicity}

Let us now bring together the statements from Theorem \ref{Thm=Khintchine} and Section \ref{Averaging operators}.

\begin{theorem} \label{simple}
Let $\left(W,S\right)$ be a right-angled, irreducible Coxeter system with $3\leq\left|S\right|<\infty$. Then there exists an open neighborhood  ${\mathcal U}\subseteq\mathbb{R}_{>0}^{(W,S)}$ of $1 =(1_s)_{s\in S}$ such that for all $q\in{\mathcal U}$ the C$^\ast$-algebra $C_{r,q}^{\ast}\left(W\right)$ is simple and has unique tracial state.
\end{theorem}

\begin{proof}
As $(W,S)$ is irreducible with $\left|S\right| \geq 3$ we find elements $s,t_0,\ldots,t_n \in S$ with $n \geq 1$, $S=\left\{s,t_0 ,\ldots,t_n\right\}$, $m_{s,t_0}=\infty$, $t_1 \neq s$ and $m_{t_i,t_{i+1}}=\infty$ for $i=0,\ldots,n-1$. Then $\mathbf{w}_1:=t_0 \cdots t_n \cdots t_0$, $\mathbf{w}_2:=s$, $\mathbf{w}_3:=t_1$ and $D:=\left\{ \mathbf{w}\in W\mid\left|t_{0}\mathbf{w}\right|<\left|\mathbf{w}\right|\right\} $ satisfy the conditions from Proposition \ref{operator norm}.
For every reduced expression $\mathbf{w}=s_{1}\cdots s_{m}$ in $W$ consider the operator
\begin{eqnarray}
\nonumber
\prod_{i=1}^{m}\left(\frac{1-q_{s_{i}}}{1+q_{s_{i}}}+\frac{2\sqrt{q_{s_{i}}}}{1+q_{s_{i}}}T_{s_{i}}^{\left(q\right)}\right)\in C_{r,q}^{\ast}\left(W\right)\text{.}
\end{eqnarray}
By the same arguments as in the proof of Proposition \ref{isomorphism} this operator is unitary and does not depend on the reduced expression for $\mathbf{w}$. By abuse of notation we will denote it by $\pi_{q,1}\left(T_{\mathbf{w}}^{\left(1\right)}\right)$. Choose a positive integer $d$ with $\left|\mathbf{w}\right|\leq d$ for all $\mathbf w$ in $F=\left\{ e,\mathbf{w}_1,\mathbf{w}_2,\mathbf{w}_3\right\}$ and define a \emph{deformed} averaging operator $\Phi_{q}$ on $C_{r,q}^{\ast}\left(W\right)$ by
\begin{equation}
\nonumber
\Phi_{q}\left(x\right)=\frac{1}{\left|F\right|}\sum_{\mathbf{w}\in F}\pi_{q,1}\left(T_{\mathbf{w}^{-1}}^{\left(1\right)}\right)x\pi_{q,1}\left(T_{\mathbf{w}}^{\left(1\right)}\right)\text{.}
\end{equation}
Again, these maps are trace-preserving, unital, completely positive and they induce contractive linear operators $\tilde{\Phi}_{q}$ on $\ell^{2}(W)\ominus\mathbb{C}\delta_{e}$ via $\tilde{\Phi}_{q}\left(x\delta_{e}\right):=\Phi_{q}\left(x\right)\delta_{e}$ for $x\in\mathbb{C}_{q}\left[W\right]$. One easily checks that $\Vert \tilde{\Phi}_{q}\Vert \rightarrow\Vert \tilde{\Phi}\Vert$ for $q \rightarrow 1$. In particular, Proposition \ref{operator norm} implies that there exists an open neighborhood ${\mathcal U}\subseteq\mathbb{R}_{>0}^{(W,S)}$ of $1 = (1_s)_{s \in S}$ such that $\tilde{\Phi}_{q}$ has norm strictly smaller than $1 \in \mathbb{R}$ for all $q\in{\mathcal U}$.

Denote by $\chi_{d}$ the word length projection on $C_{r,q}^\ast (W)$ from equation \eqref{Eqn=WordLengthProj} and put $\chi_{\leq d}:=\sum_{r=0}^{d}\chi_{r}$. Let $l \geq 1$ and $x\in\chi_{\leq d}\left(\mathbb{C}_{q}\left[W\right]\right)$. Note that for every $\mathbf{w}\in F$ we have $\chi_r({\pi_{q,1}(T_\mathbf{w}^{(1)})})=0$ for $r \geq d$. As $\Phi_q^l$ averages over $\left\{ \pi_{q,1}(T_{\mathbf{w}}^{\left(1\right)})\mid\mathbf{w}\in F\right\} $ and $\chi_r(x)=0$ for $r \geq d$, we get that $\chi_r(\Phi_q^l(x))=0$ for $r \geq (2l+1)d$. In particular
\begin{eqnarray}
\nonumber
\Vert \Phi_{q}^{l}\left(x\right)-\tau_q\left(x\right)\Vert \leq \sum_{r=1}^{(2l+1)d}\Vert \chi_{r}\left(\Phi_{q}^{l}\left(x\right)\right)\Vert
\end{eqnarray}
by the triangle inequality and with $C:=(\# \Cliq(\Gamma))^3 ( \prod_{s \in S} p_s(q) )$ we have
\begin{eqnarray}
\nonumber
\Vert \Phi_{q}^{l}\left(x\right)-\tau_q\left(x\right)\Vert \leq \sum_{r=1}^{(2l+1)d}Cr\Vert \chi_{r}\left(\Phi_{q}^{l}\left(x\right)\right)\Vert _{2}
\end{eqnarray}
by Theorem \ref{Thm=Khintchine}. The Cauchy-Schwarz inequality then implies
\begin{eqnarray}
\nonumber
\Vert \Phi_{q}^{l}\left(x\right)-\tau_q\left(x\right)\Vert &\leq& \left((2l+1)d\right)^{\frac{1}{2}}\left(\sum_{r=1}^{(2l+1)d}C^{2}r^{2}\Vert \chi_{r}\left(\Phi_{q}^{l}\left(x\right)\right)\Vert _{2}^{2}\right)^{\frac{1}{2}}\\
\nonumber
 &\leq& C\left((2l+1)d\right)^{\frac{3}{2}}\left(\sum_{r=1}^{(2l+1)d}\Vert \chi_{r}\left(\Phi_{q}^{l}\left(x\right)\right)\Vert _{2}^{2}\right)^{\frac{1}{2}}\\
\nonumber
&=& C\left((2l+1)d\right)^{\frac{3}{2}}\Vert \Phi_{q}^{l}\left(x\right)-\tau_q\circ\Phi_{q}^{l}\left(x\right)\Vert _{2}\\
\nonumber
&\leq& C\left((2l+1)d\right)^{\frac{3}{2}}\Vert \tilde{\Phi}_{q}\Vert ^{l}\Vert x-\tau_q\left(x\right)\Vert _{2} \text{.}
\end{eqnarray}
 For $q\in{\mathcal U}$ this converges to $0$ as $l \rightarrow \infty$. The simplicity and the unique trace property of $C_{r,q}^{\ast}\left(W\right)$, $q\in{\mathcal U}$ now follow by a standard argument (see for instance \cite{Powers}): Let $I$ be a non-zero ideal in $C_{r,q}^{\ast}\left(W\right)$. Choose $0\neq x\in I$ positive. For every $\varepsilon>0$ we find $x_{\varepsilon}\in\mathbb{C}_{q}\left[W\right]$ with $\left\Vert x-x_{\varepsilon}\right\Vert <\varepsilon/3$. For $l$ large enough this implies
\begin{eqnarray}
\label{ideal}
\Vert \Phi_{q}^{l}\left(x\right)-\tau_q\left(x\right)\Vert \leq\Vert \Phi_{q}^{l}\left(x-x_{\varepsilon}\right)\Vert +\Vert \Phi_{q}^{l}\left(x_{\varepsilon}\right)-\tau_q\left(x_{\varepsilon}\right)\Vert +\Vert \tau_q\left(x_{\varepsilon}\right)-\tau_q\left(x\right)\Vert <\varepsilon\text{,}
\end{eqnarray}
so $0\neq\tau_q\left(x\right)\in I$. Hence $I$ must be trivial and this shows that $C_{r,q}^{\ast}\left(W\right)$ is simple if $q\in{\mathcal U}$. Further, as $\Phi_{q}^{l}\left(x\right)\rightarrow\tau_q\left(x\right)$ by \eqref{ideal} and $\tau'\left(\Phi_{q}^{l}\left(x\right)\right)=\tau'\left(x\right)$ for every tracial state $\tau'$, $C_{r,q}^{\ast}\left(W\right)$ has $\tau_q$ as its unique tracial state.
\end{proof}

\begin{remark}\label{Rmk=Factor}$\:$ \\
\emph{(a)} Theorem \ref{simple} immediately implies that for $q\in{\mathcal U}$ the Hecke-von Neumann algebra ${\mathcal N}_{q}\left(W\right)$ is a $\text{II}_{1}$-factor  by uniqueness of the trace and Lemma \ref{Lem=StandardExt}. This extends results by Garncarek \cite{Gar} to the multi-parameter case; note that the proof from \cite{Gar} does not trivially extend to the multi-parameter case. {\it Note:} After completion of this paper the factoriality question in the right-angled multi-parameter case was fully solved by Raum and Skalski in \cite{RaumSkalski}.
\newline
\emph{(b)} Theorem \ref{simple} does not give any information about the simplicity of Hecke C$^{\ast}$-algebras of infinitely generated Coxeter groups unless $q$ equals $1$.
\end{remark}


\subsection{Free products of abelian Coxeter groups}\label{Sect=FreeAbelian}

In this section we want to investigate the maximum size of the open neighborhood ${\mathcal U}\subseteq\mathbb{R}_{>0}^{(W,S)}$ in Theorem \ref{simple}. Though our aim is different, the approach is inspired by \cite[Section 6]{Gar} and translates the remarks made there into the C$^{\ast}$-algebraic setting by using results from \cite{Dykema}.

Let us assume that $\left(W,S\right)$ is a Coxeter system where $W$ is of the form $W=\mathbb{Z}_{2}^{k_{1}}\ast\cdots\ast\mathbb{Z}_{2}^{k_{l}}$ with $l, k_{1}\geq2$ and $k_{2},\ldots,k_{l}\in\mathbb{N}$. For each $1 \leq m \leq l$ denote by $s_{1}^{\left(m\right)},\ldots,s_{k_{m}}^{\left(m\right)}$ the mutually commuting generators corresponding to the component $\mathbb{Z}_{2}^{k_{m}}$ of $W$ and set $S_{m}:=\left\{ s_{1}^{\left(m\right)},\ldots,s_{k_{m}}^{\left(m\right)}\right\}$, so in particular $S=\bigcup_{m=1}^{l}S_{m}$. Let $q \in \mathbb{R}_{>0}^{(W,S)}$. The Hecke C$^{\ast}$-algebra $C_{r,q}^{\ast}\left(W\right)$ decomposes as a reduced free product
\begin{eqnarray} \nonumber
\left(C_{r,q}^{\ast}\left(W\right),\tau_q\right) \simeq \ast_{m=1}^{l}\left(\bigotimes_{i=1}^{k_{m}}\left(C_{r,q}^{\ast}\left(W_{s_{i}^{\left(m\right)}}\right),\tau_q\right)\right)
\end{eqnarray}
over the canonical traces. In the notation of \cite{Dykema} we get using \cite[Lemma 6.2]{Gar} that
\begin{eqnarray} \label{decomposition}
 \left(C_{r,q}^{\ast}\left(W\right),\tau_q\right) \simeq \ast_{m=1}^{l}\left(C\left(\mathbb{Z}_{2}^{k_{m}}\right),\int(\cdot)\:\mu_{m}\right),
\end{eqnarray}
where $\mu_{m}\left(\mathbf{w}\right)=q_{\mathbf w}\prod_{i=1}^{k_{m}}\left(1+q_{s_{i}^{\left(m\right)}}\right)^{-1}$ for $\mathbf{w}\in\mathbb{Z}_{2}^{k_{m}}$.

\begin{proposition} \label{radius}
Let $\left(W,S\right)$ be a Coxeter system of the form $W=\mathbb{Z}_{2}^{k_{1}}\ast \cdots \ast\mathbb{Z}_{2}^{k_{l}}$ where $l, k_{1},\ldots ,k_{l}\in\mathbb{N}$ and $q \in \mathbb{R}_{>0}^{(W,S)}$. Then the growth series $W\left(z\right)$ of $W$ is equal to the Taylor expansion of the multivariate function
\begin{eqnarray} \label{multivariate}
z\mapsto\left(\sum_{m=1}^{l}\prod_{i=1}^{k_{m}}\left(1+z_{s_{i}^{\left(m\right)}}\right)^{-1}-\left(l-1\right)\right)^{-1}\text{.}
\end{eqnarray}
The series converges for every element in
\begin{eqnarray}
\nonumber
\Omega:=\left\{ \left(z_{s}\right)_{s\in S}\in \mathbb{C}^{(W,S)} \cap \left[0,1\right]^{S}\mid\sum_{m=1}^{l}\prod_{i=1}^{k_{m}}\left(1+z_{s_{i}^{\left(m\right)}}\right)^{-1}>l-1\right\}\text{,}
\end{eqnarray}
i.e. $\Omega\subseteq\mathcal{R}$ where $\mathcal{R}$ is the region of convergence of $W\left(z\right)$.
\end{proposition}

\begin{proof}
It is clear that the multivariate function $z\mapsto\sum_{\mathbf{w}\in W}z_{\mathbf w}$ converges absolutely on a domain $\mathcal{U}$ around $0$ and that it defines a holomorphic function there. Put $\mathcal{A}_{m}:=\left\{ \mathbf{w}\in W\mid \mathbf{w}\text{ starts in }\mathbb{Z}_{2}^{k_{m}}\right\}$ where we assume by convention that $e\in\mathcal{A}_{m}$. For every $z\in\mathcal{U}$ and $m=1,\ldots,l$ we have
\begin{eqnarray} \nonumber
W\left(z\right)&=&1+\sum_{\mathbf{w}\in\mathcal{A}_{m}\setminus\left\{ e\right\} }z_{\mathbf w}+\sum_{\mathbf{w}\in\mathcal{A}_{m}^{c}}z_{\mathbf w}\\
\nonumber
&=& 1+\sum_{\mathbf{v}\in\mathbb{Z}_{2}^{k_{m}}\setminus\left\{ e\right\} }z_{\mathbf v}+\sum_{\mathbf{v}\in\mathbb{Z}_{2}^{k_{m}}\setminus\left\{ e\right\} }z_{\mathbf v}\sum_{\mathbf{w}\in\mathcal{A}_{m}^{c}}z_{\mathbf w}+\sum_{\mathbf{w}\in\mathcal{A}_{m}^{c}}z_{\mathbf w}\\
\nonumber
&=& \left(\sum_{\mathbf{v}\in\mathbb{Z}_{2}^{k_{m}}}z_{\mathbf v}\right)\left(1+\sum_{\mathbf{w}\in\mathcal{A}_{m}^{c}}z_{\mathbf w}\right)\\
\nonumber
&=& \left(1+\sum_{\mathbf{w}\in\mathcal{A}_{m}^{c}}z_{\mathbf w}\right)\prod_{i=1}^{k_{m}}\left(1+z_{s_{i}^{\left(m\right)}}\right)\text{.}
\end{eqnarray}
We get, by using this in the third line of the next equalities, that
\begin{eqnarray}
\nonumber
W\left(z\right)&=&1+\sum_{m=1}^{l}\sum_{\mathbf{w}\in\mathcal{A}_{m}\setminus\left\{ e\right\} }z_{\mathbf w}\\
\nonumber
&=& 1+\sum_{m=1}^{l}\left(W\left(z\right)-1-\sum_{\mathbf{w}\in\mathcal{A}_{m}^{c}}z_{\mathbf w}\right)\\
\nonumber
&=& 1+\sum_{m=1}^{l}\left(1-\prod_{i=1}^{k_{m}}\left(1+z_{s_{i}^{\left(m\right)}}\right)^{-1}\right)W\left(z\right)
\end{eqnarray}
and hence
\begin{eqnarray}\label{Eqn=WTaylor}
W\left(z\right)=\left(\sum_{m=1}^{l}\prod_{i=1}^{k_{m}}\left(1+z_{s_{i}^{\left(m\right)}}\right)^{-1}-\left(l-1\right)\right)^{-1}\text{.}
\end{eqnarray}
This implies that the growth series $W\left(z\right)$ of $W$ is equal to the Taylor expansion of \eqref{multivariate} in $0$. From \eqref{Eqn=WTaylor} we also find that \begin{equation}\label{Eqn=QDec}
W\left(z\right)=\left(1-Q\left(z\right)\right)^{-1}\prod_{s\in S}\left(1+z_{s}\right)
\end{equation}
 with the polynomial
\begin{eqnarray} \nonumber
Q\left(z\right):=1+\left(l-1\right)\prod_{s\in S}\left(1+z_{s}\right)-\sum_{m=1}^{l}\prod_{n\neq m}\prod_{j=1}^{k_{n}}\left(1+z_{s_{j}^{\left(n\right)}}\right)\text{.}
\end{eqnarray}
One has $Q\left(0\right)=0$. With $D^\alpha Q$ denoting the higher order partial derivative of $Q$ with respect to $0 \neq \alpha \in \left\{0,1\right\}^S$ we have
\begin{eqnarray} \nonumber
\left.D^{\alpha}Q\right|_{z=0}&=&\left[\left(l-1\right)\prod_{s\in S\text{: }\alpha_{s}=0}\left(1+z_{s}\right)-\sum_{m \in J} \prod_{n \neq m} \prod_{j \in K_n}\left(1+z_{s_{j}^{\left(n\right)}}\right)\right]_{z=0}\\
\nonumber
&=& \left(l-1\right)-\sum_{m \in J}1\\
\nonumber
&\geq& 0 \text{,}
\end{eqnarray}
where $J:=\left\{m \mid 1\leq m \leq l \text{ with } \alpha_s=0 \text{ for all } s\in S_m \right\}$ and
\[
K_n :=\left\{ j \mid 1 \leq j \leq k_m \text{ with } \alpha_{s_j^{(n)}}=0 \right\}\text{.}
\]
This implies that $Q$ has only positive coefficients. Hence for $z \in \Omega$ we have $Q(z) \geq 0$.  Moreover, for $z\in\Omega$ we have $W(z) > 0$ and so by \eqref{Eqn=QDec} we have $0\leq Q\left(z\right)<1$.

Further, for any $z\in \mathbb{C}^{(W,S)}$ with $0\leq Q\left(z\right)<1$ we can expand and rearrange the terms in the series $\sum_{m=0}^{\infty}\left(Q\left(z\right)\right)^{m}$ to get the (converging) Taylor series of $\left(1-Q\left(z\right)\right)^{-1}$. The same is true for the product $\left(1-Q\left(z\right)\right)^{-1}\prod_{s\in S}\left(1+z_{s}\right)$. Combining this with the previous paragraph we get that the Taylor series of the function in \eqref{multivariate} (which is the growth series $W\left(z\right)$) converges on $\Omega$.
\end{proof}

Taking into account \cite[Corollary 4.10]{Dykema}, Proposition \ref{radius} implies the following.

\begin{theorem} \label{free product}
Let $\left(W,S\right)$ be a Coxeter system of the form $W=\mathbb{Z}_{2}^{k_{1}}\ast \cdots \ast\mathbb{Z}_{2}^{k_{l}}$ where $l,k_{1}\geq2$ and let $q \in \mathbb{R}_{>0}^{(W,S)}$. Then the following statements are equivalent.

\begin{enumerate}
\item $C_{r,q}^{\ast}\left(W\right)$ is simple and has unique tracial state;
\item $q\notin\overline{\mathcal{R}'}$, where $\mathcal{R}'$ is defined in \eqref{generalized radius}.
\end{enumerate}
\end{theorem}

\begin{proof}
$\left(1\right)\Rightarrow\left(2\right)$: Assume that $q\in\overline{\mathcal{R}'}$. Then $C_{r,q}^{\ast}\left(W\right)$ is not simple and does not have unique tracial state by Lemma \ref{non-simplicity1}.

$\left(2\right)\Rightarrow\left(1\right)$: Assume that $C_{r,q}^{\ast}\left(W\right)$ is not simple and does not have unique tracial state. Assume further that $0<q_{s} \leq 1$ for every $s\in S$. Using \cite[Proposition 4.10]{Dykema} in combination with \eqref{decomposition} we get that the set
\begin{eqnarray} \nonumber
\left\{ \left(\mathbf{w}_{1},...,\mathbf{w}_{l}\right)\in\prod_{m=1}^{l}\mathbb{Z}_{2}^{k_{m}}\mid l-1\leq \sum_{m=1}^{l}q_{\mathbf{w}_m}\prod_{i=1}^{k_{m}}\left(1+q_{s_{i}^{\left(m\right)}}\right)^{-1}\right\}
\end{eqnarray}
is not empty, so in particular
\begin{eqnarray} \nonumber
l-1\leq\max_{\left(\mathbf{w}_{1},...,\mathbf{w}_{l}\right)\in\prod_{m=1}^{l}\mathbb{Z}_{2}^{k_{m}}}\left(\sum_{m=1}^{l}q_{\mathbf{w}_{m}}\prod_{i=1}^{k_{m}}\left(1+q_{s_{i}^{\left(m\right)}}\right)^{-1}\right)\leq\sum_{m=1}^{l}\prod_{i=1}^{k_{m}}\left(1+q_{s_{i}^{\left(m\right)}}\right)^{-1}.
\end{eqnarray}
Comparing this with Proposition \ref{radius}, we get that $q\in\overline{\mathcal{R}'}$. The general case of $q \in \mathbb{R}_{>0}^{(W,S)}$ follows by an application of Proposition \ref{unitary}.
\end{proof}

\begin{remark}\label{Rmk=Factor2}
Proposition \ref{radius} also implies that the reasoning in \cite[Section 6]{Gar} applies to the multi-parameter Hecke-von Neumann algebra $\mathcal{N}_{q}\left(W\right)$ with $W$ as in Theorem \ref{free product}. In other words, $\mathcal{N}_{q}\left(W\right)$ is a factor if and only if $q\notin\mathcal{R}'$.  {\it Note:} After completion of this paper this result was extended to arbitrary irreducible right-angled Hecke von Neumann algebras with $\vert S \vert \geq 3$  in \cite{RaumSkalski}.
\end{remark}

\begin{remark}
In the proof of Theorem \ref{free product} we only used the simplicity part of \cite[Corollary 4.10]{Dykema}. For $W$ as above the full statement of the corollary also provides a detailed description of the ideal structure of $C_{r,q}^{\ast}\left(W\right)$ for $q\in\overline{\mathcal{R}'}$. Further we conclude by \cite[Corollary 4.10]{Dykema} that for $l \geq 3$ the Hecke C$^\ast$-algebra $C_{r,q}^\ast(W)$ has stable rank $1$ for every $q$.
\end{remark}

In view of Lemma \ref{non-simplicity1}, Proposition \ref{unitary}, Theorem \ref{simple} and Theorem \ref{free product} it is a natural question whether or not the characterisation above holds for general Coxeter systems.

\begin{question}
Let $\left(W,S\right)$ be an irreducible Coxeter system and $q \in \mathbb{R}_{>0}^{(W,S)}$. Is it true that the Hecke C$^{\ast}$-algebra $C_{r,q}^{\ast}\left(W\right)$ is simple and has unique tracial state if and only if $q\notin\overline{\mathcal{R}'}$ where $\mathcal{R}'$ is defined as in \eqref{generalized radius}?
\end{question}

\begin{remark}
Recently a new approach to C$^\ast$-simplicity results was obtained through Furstenberg/Hamana boundaries, see \cite{KK}, \cite{BKKO}, \cite{KB}. In particular, in \cite{KB} the Furstenberg boundary of a general unitary representation of a discrete group was defined and investigated in relation to trace-uniqueness properties. Proposition \ref{isomorphism} shows that Hecke C$^\ast$-algebras of a right-angled Coxeter group are C$^\ast$-algebras generated by such a unitary representation. It would be interesting to exploit this connection. However, in light of the results from \cite{KB}, it is not clear  how manageable the Furstenberg-Hamana boundary is.
\end{remark}


\section{Exactness and nuclearity of   Hecke C$^\ast$-algebras} \label{FurtherProperties}

The aim of this last section is to have a look at additional properties of Hecke C$^{\ast}$-algebras such as exactness and nuclearity and give (counter-)examples regarding the extension of the isomorphism $\pi_{q,1}$ in Proposition \ref{isomorphism} to the C$^\ast$-algebraic level.

\begin{theorem}
Let $\left(W,S\right)$ be a Coxeter system and $q=\left(q_{s}\right)_{s\in S}\in\mathbb{R}_{>0}^{(W,S)}$. Then $C_{r,q}^{\ast}\left(W\right)$ is an exact C$^{\ast}$-algebra.
\end{theorem}

\begin{proof}
By \cite{DJ} (see also \cite{Anantharaman}) the group $W$ acts amenably on a compact space. Therefore, by \cite[Theorem 5.1.7]{BrownOzawa} $W$ is exact. With $P_{s}$ as defined in \eqref{Eqn=ProjectionSpace}, we have $T_{s}^{\left(q\right)}=T_{s}^{\left(1\right)}+p_{s}(q)P_{s}$ for $s\in S$. The Hecke C$^{\ast}$-algebra $C_{r,q}^{\ast}\left(W\right)$ is generated by $\left\{ T_{s}^{\left(q\right)}\right\} _{s\in S}$ and we have $P_{s}\in \ell^{\infty}(W)\subseteq{\mathcal B}\left(\ell^2(W)\right)$. Hence $C_{r,q}^{\ast}\left(W\right)$ is contained in the uniform Roe algebra of $W$ which is nuclear by \cite[Theorem 5.1.6]{BrownOzawa}. That implies the exactness of $C_{r,q}^{\ast}\left(W\right)$.
\end{proof}

The following theorem generalizes \cite[Theorem 3.6]{CaspersAPDE}.

\begin{theorem} \label{equivalence}
Let $\left(W,S\right)$ be an irreducible Coxeter system and $q \in \mathbb{R}_{>0}^{(W,S)}$. Then the following statements are equivalent:
\begin{enumerate}[label=(\arabic*)]
\item $\left(W,S\right)$ is of spherical or affine type;
\item $C_{r,q}^{\ast}\left(W\right)$ is nuclear for all $q$;
\item $C_{r,q}^{\ast}\left(W\right)$ is nuclear for some $q$;
\item ${\mathcal N}_{q}\left(W\right)$ is injective for all $q$;
\item ${\mathcal N}_{q}\left(W\right)$ is injective for some $q$.
\end{enumerate}

\end{theorem}

\begin{proof}
$\left(1\right)\Rightarrow\left(2\right)$: If $\left(W,S\right)$ is of spherical type, then $C_{r,q}^{\ast}\left(W\right)$ is  finite-dimensional or an inductive limit over finite-dimensional C$^{\ast}$-algebras, and in either case nuclear. So let us assume that $\left(W,S\right)$ is of affine type. It is well-known that affine Coxeter systems correspond to affine Weyl groups, see for example \cite[Chapter VI]{B}. From the discussion in \cite[\S4]{Matsumoto} it follows that the Hecke C$^{\ast}$-algebra $C_{r,q}^{\ast}\left(W\right)$ must be of type I. In particular it is nuclear,  see \cite{TakesakiTo}.

$\left(2\right)\Rightarrow\left(3\right)$: Clear.

$\left(3\right)\Rightarrow\left(5\right)$: If $C_{r,q}^{\ast}\left(W\right)$ is nuclear for some $q$, then the bicommutant ${\mathcal N}_{q}\left(W\right)=\left(C_{r,q}^{\ast}\left(W\right)\right)''$ must be injective \cite[Exercise 3.6.4, Corollary 3.8.6 and Theorem 9.3.3]{BrownOzawa}.

$\left(5\right)\Rightarrow\left(1\right):$ Let $(W,S)$ be an irreducible Coxeter system of non-affine type. By \cite[Proposition 17.2.1]{Da}, $W$ contains the free group on two generators. Denote the corresponding generators by $\mathbf{a}_{1}$, $\mathbf{a}_{2}$ and let ${\mathcal M}$ be the von Neumann subalgebra of ${\mathcal N}_{q}\left(W\right)$ generated by $T_{\mathbf{a}_{1}}^{\left(q\right)}$ and $T_{\mathbf{a}_{2}}^{\left(q\right)}$. Let further ${\mathcal M}_{1}$ be the von Neumann algebra generated by $T_{\mathbf{a}_{1}}^{\left(q\right)}$ and ${\mathcal M}_{2}$ the one generated by $T_{\mathbf{a}_{2}}^{\left(q\right)}$. Then ${\mathcal M}$ is isomorphic to the free product $\left({\mathcal M}_{1},\tau_{q,1}\right)\ast\left({\mathcal M}_{2},\tau_{q,2}\right)$ over the canonical traces. The dimensions of ${\mathcal M}_{1}$, ${\mathcal M}_{2}$ are infinite, so ${\mathcal M}$ is non-injective by \cite[Theorem 4.1 and Remark 4.2 (5)]{Ueda}. But there exists a trace-preserving normal conditional expectation $E\text{: }{\mathcal N}_{q}\left(W\right)\twoheadrightarrow{\mathcal M}$ (see for example \cite[Lemma 1.5.11]{BrownOzawa} or \cite[Corollary 3.3]{CaspersAPDE}), so ${\mathcal N}_{q}\left(W\right)$ must be non-injective as well.

$ \left(2\right) \Leftrightarrow\left(4\right):$ Clear from the arguments above.
\end{proof}

We have seen that for spherical type Coxeter systems the corresponding Hecke C$^{\ast}$-algebras are independent of the choice of $q$. Further, if the Coxeter system $\left(W,S\right)$ is right-angled and amenable, then the map $\pi_{q,1}$ from Proposition \ref{isomorphism} extends to a surjection $C_{r}^{\ast}\left(W\right)\twoheadrightarrow C_{r,q}^{\ast}\left(W\right)$ for every $q$. It is a natural question whether or not this map is an isomorphism as well. This is not always the case as the following example illustrates.

\begin{example} \label{example1}
Let $\left(W,S\right)$ be the Coxeter system generated by two elements $s$, $t$ with $m_{s,s}=m_{t,t}=2$, $m_{s,t}=\infty$. This is the infinite dihedral group which is the only irreducible right-angled Coxeter group of affine type. Let $C_{r,0}^{\ast}\left(W\right)\subseteq{\mathcal B}\left(\ell^2(W)\right)$ be the unital C$^{\ast}$-algebra generated by $P_{s}$ and $P_{t}$ defined in \eqref{Eqn=ProjectionSpace}. Assume that $\pi_{q,1}$ of Proposition \ref{isomorphism} extends to a $\ast$-isomorphism for every $q\in\mathbb{R}_{>0}^{(W,S)}$. Then the map $P_{u}\mapsto(1-T_{u}^{(1)})/2$, $u\in S$ extends to a $\ast$-isomorphism $\pi_{0,1}\text{: }C_{r,0}^{\ast}\left(W\right) \simeq C_{r}^{\ast}\left(W\right)$ as well. Indeed, using Proposition \ref{isomorphism} one checks that as $q \downarrow 0$ we have
\begin{equation}
\nonumber
\pi_{q,1}^{-1}\left(x\right)\rightarrow\sum_{\mathbf{w}\in W}x\left(\mathbf{w}\right)\pi_{0,1}^{-1}\left(T_{\mathbf w}^{\left(1\right)}\right)=\pi_{0,1}^{-1}\left(x\right)
\end{equation}
in ${\mathcal B}\left(\ell^2(W)\right)$ for every $x:=\sum_{\mathbf{w}\in W}x\left(\mathbf{w}\right)T_{\mathbf w}^{\left(1\right)}\in\mathbb{C}_{1}\left[W\right]$. Hence,
\begin{equation}
\nonumber
\Vert \pi_{0,1}^{-1}\left(x\right)\Vert =\lim_{q\downarrow0} \Vert \pi_{q,1}^{-1}\left(x\right)\Vert =\Vert x\Vert
\end{equation}
as $\pi_{q,1}$ is isometric. Since the C$^{\ast}$-algebra $C_{r,0}^{\ast}\left(W\right)$ is commutative, we have reached a contradiction.
\end{example}

The example illustrates that $\pi_{q,1}$ in general does not extend to an isomorphism of (reduced) C$^{\ast}$-algebras. In the non-affine case it gets even worse. Here $\pi_{q,1}$ and $\pi_{q,1}^{-1}$ never extend to the reduced C$^{\ast}$-algebra level for arbitrary choice of $q$.

\begin{example}
Let $\left(W,S\right)$ be an irreducible right-angled Coxeter system of non-affine type. By Theorem \ref{simple} the reduced group C$^{\ast}$-algebra $C_{r}^{\ast}\left(W\right)$ has unique tracial state. The map $\pi_{q,1}$ is not trace preserving with respect to the canonical trace. Hence it does not extend to the reduced C$^{*}$-algebraic level. To see that $\pi_{q,1}^{-1}$ does in general not extend, one can use the same argument as in Example \ref{example1} or the trace-uniqueness in Theorem \ref{simple}.
\end{example}


\end{document}